\documentclass[11pt, reqno]{amsart}
\usepackage{amscd}
\usepackage{amsmath}
\usepackage{amssymb}
\usepackage{latexsym}
\usepackage{amsthm}
\usepackage{tikz,etex,xspace}
\usepackage{graphicx}
\usepackage[all]{xy}       

    \SelectTips{cm}{10}     

    \everyxy={<2.5em,0em>:} 

    \xyoption{web}          

\usepackage{hyperref}

\newcommand{\arxiv}[1]{\href{http://arxiv.org/abs/#1}{\tt arXiv:\nolinkurl{#1}}}
\newcommand{\arXiv}[1]{\href{http://arxiv.org/abs/#1}{\tt arXiv:\nolinkurl{#1}}}

\newtheorem{theorem}{Theorem}[section]
\newtheorem{lemma}[theorem]{Lemma}

\newtheorem{definition}[theorem]{Definition}
\newtheorem{example}[theorem]{Example}

\newtheorem{proposition}[theorem]{Proposition}
\newtheorem{corollary}[theorem]{Corollary}

\theoremstyle{remark}
\newtheorem{remark}[theorem]{Remark}

\numberwithin{equation}{section}

\newcommand{\nc}{\newcommand}
\newcommand{\dct}[1]{{\underset{#1}{\bigcirc}} }
\newcommand{\ic}[1]{{\circ_{#1}}} 
\nc{\hd}{\operatorname{hd}}

\nc{\flags}{\mathcal{F}}
\nc{\om}{\omega}
\nc{\KP}{\operatorname{KP}}
\nc{\Om}{\Omega}
\nc{\be}{\begin{enumerate}}
\nc{\bnum}{\be[{\rm(i)}]}

\def\dual{^\circledast}
\nc{\w}{\omega}
\def\ii{{\bf i}}
\nc{\re}{{re}}

\nc{\uline}{\underline}
\nc{\ula}{{\underline{\la}}}
\nc{\umu}{{\underline{\mu}}}
\nc{\unu}{{\underline{\nu}}}
\nc{\ra}{\rangle}

\def\N{\mathbb{N}}
\def\Q{\mathbb{Q}}\def\partition{\mathcal{P}}

\def\C{\overline{\mathbb{Q}}}
\def\R{\mathbb{R}}
\def\Z{\mathbb{Z}}

\def\A{\mathcal{A}}
\def\W{\Omega}
\def\B{\mathcal{B}}
\def\ka{\kappa}

\def\uj{\bf j}
\def\jj{{\bf j}}

\def\P{\mathcal{P}} 

\def\Proj{\operatorname{Proj}}

\nc{\co}{\nabla}

\def\a{\alpha}
\def\b{\beta}

\def\la{\lambda}
\def\La{\Lambda}
\def\ga{\gamma}

\def\d{\delta}
\def\de{\delta}
\def\th{\theta}
\def\f{\mathbf{f}}
\def\g{\mathfrak{g}}

\def\D{\Delta}

\def\Hom{\operatorname{Hom}}
\def\dimq{\dim_q}
\def\Ind{\operatorname{Ind}}
\def\Coind{\operatorname{CoInd}}
\def\seq{\,\mbox{Seq}\,}

\def\Mat{\,\mbox{Mat}\,}
\def\End{\operatorname{End}}
\def\Seq{\,\mbox{Seq}\,}

\def\Res{\operatorname{Res}}

\def\Ext{\operatorname{Ext}}
\def\coker{\operatorname{coker}}

\def\mods{\mbox{-mod}}
\def\fmod{\mbox{-fmod}}
\def\prmod{\mbox{-pmod}}

\newcommand{\map}[2]{\,{:}\,#1\!\to\!#2}

\def\p{\pi} %

\def\inv{^{-1}}

\numberwithin{equation}{section}

\title[KLR algebras]{Representations of Khovanov-Lauda-Rouquier algebras III: Symmetric Affine Type}
\address{}\email{maths@petermc.net}
\author{Peter J McNamara}
\date{\today}

\begin{document}

\begin{abstract}
We develop the homological theory of KLR algebras of symmetric affine type. For each PBW basis, a family of standard modules is constructed which categorifies the PBW basis.
\end{abstract}
\maketitle
\tableofcontents
\section{Introduction} 

Khovanov-Lauda-Rouquier algebras (henceforth KLR algebras), also known as Quiver Hecke algebras, are a family of $\Z$-graded associative algebras introduced by Khovanov and Lauda \cite{khovanovlauda} and Rouquier \cite{rouquier} for the purposes of categorifying quantum groups. More specifically they categorify the upper-triangular part $\f=U_q(\mathfrak{g})^+$ of the quantised enveloping algebra of a symmetrisable Kac-Moody Lie algebra $\g$ - see \S \ref{klr} for a precise statement. Let $I$ be the set of simple roots of $\g$ and $\N I$ the monoid of formal sums of elements of $I$. For each $\nu\in\N I$ there is an associated KLR algebra $R(\nu)$.

In this paper we will assume that $\g$ is of symmetric affine type. For now however, we will describe the theory developed in \cite{klr1,bkm} where $\g$ is finite dimensional. The results of this paper generalise these results to the symmetric affine case.

One begins with choosing a convex order $\prec$ on the set of positive roots satisfying a convexity property - see Definition \ref{convexdefn}. It is this convex order which determines a PBW basis of $\f$. The representation theory of KLR algebras is built via induction functors from the theory of cuspidal representations. Write $\{\a_1\succ \cdots \succ\a_N\}$ for the set of positive roots, remembering that we are temporarily discussing the finite type case.

To each root $\a$ there is a subcategory of $R(\a)$-modules which are cuspidal defined in Definition \ref{cuspidaldefn}. There is a unique irreducible cuspidal module $L(\a)$. Let $\D(\a)$ be the projective cover of $L(\a)$ in the category of cuspidal $R(\a)$-modules.

Given any sequence $\p=(\p_1,\ldots,\p_N)$ of natural numbers, the proper standard and standard modules are defined respectively by
\begin{align*}
 \overline{\D}(\p)&=L(\a_1)^{\circ \p_1} \circ \cdots \circ L(\a_N)^{\circ \p_N} \\
\D(\p)&= \D(\a_1)^{(\p_1)}\circ \cdots\circ \D(\a_N)^{(\p_N)}
\end{align*}
where $\circ$ denotes the induction of a tensor product and $(\p_i)$ is a divided power construction. Then in \cite{klr1} it is proved that the modules $\overline{\D}(\p)$ categorify the dual PBW basis, have a unique irreducible quotient and that these quotients give a classification of all irreducible modules. In \cite{bkm} it is proved that the modules $\D(\p)$ categorify the PBW basis and their homological properties are studied, justifying the use of the term standard.

Now let us turn our attention to the results of this paper where $\g$ is of symmetric affine type. Again the starting point is the choice of a convex order $\prec$ on the set of positive roots. The theory of PBW bases for affine quantised enveloping algebras dates back to the work of Beck \cite{beck} and is considerably more complicated than the theory in finite type. It is a feature of the literature that the theory of PBW bases is only developed for convex orders of a  particular form. We rectify this problem by presenting a construction of PBW bases in full generality.

For $\a$ a real root, the category of cuspidal $R(\a)$-modules is again equivalent to the category of $k[z]$-modules while the category of semicuspidal $R(n\a)$-modules is again equivalent to modules over a polynomial algebra. Whereas in finite type the proofs of these results currently rest on some case by case computations, here we give a uniform proof, the cornerstone of which is the growth estimates in \S \ref{growth}.

For the imaginary roots, the category of semicuspidal representations is qualitatively very different. The key observation here is that the R-matrices constructed by Kang, Kashiwara and Kim \cite{kkk} enable us to determine an isomorphism
\[
 \End(M^{\circ n})_0\cong \Q[S_n]
\]
where $M$ is either an irreducible cuspidal $R(\d)$-module or an indecomposable projective in the category of cuspidal $R(\d)$-modules (here $\d$ is the minimal imaginary root). We are then able to use the representation theory of the symmetric group to decompose these modules $M^{\circ n}$. This presence of the symmetric group as an endomorphism algebra can also be seen to explain the appearance of Schur functions in the definition of a PBW basis in affine type.

With the semicuspidal modules understood we are able to prove our main theorems which are analogous to those discussed above in finite type. Namely families of proper standard and standard modules are constructed which categorify the dual PBW and PBW bases respectively. See Theorem \ref{pbwcat} and the following paragraph for this result.
Compared with the corresponding theorem in \cite[Proposition 4.11]{kleshchev}, we are able to identify the imaginary constituents of the PBW basis.
The proper standard modules have a unique irreducible quotient which gives a classification of all irreducibles and the standard modules satisfy homological properties befitting their name, leading to a BGG reciprocity theorem.

As a consequence we obtain a new positivity result, Theorem \ref{positif}, which states that when an element of the canonical basis of $\f$ is expanded in a PBW basis, the coefficients that appear are polynomials in $q$ and $q\inv$ with non-negative coefficients (and the transition matrix is unitriangular).


%
%

We thank A. Kleshchev, P. Tingley and B. Webster for useful conversations.

\section{Preliminaries}

The purpose of this section is to collect standard notation about root systems and other objects which we will be making use of in this paper.

Let $(I,\cdot)$ be a Cartan Datum of symmetric affine type.
Following the approach of Lusztig \cite{lusztigbook}, this comprises a finite set $I$ and a symmetric pairing $\cdot\map{I\times I}{\Z}$ such that $i\cdot i=2$ for all $i\in I$, $i\cdot j\leq 0$ if $i\neq j$ and the matrix $(i\cdot j)_{i,j\in I}$ is of corank 1.
Such Cartan data are completely classified and correspond to the extended Dynkin diagrams of type A, D and E. We extend $\cdot\map{I\times I}{\Z}$ to a bilinear pairing $\N I\times \N I\to \Z$.

Let $\Phi^+$ be the set of positive roots in the corresponding root system.
We identify $I$ with the set of simple roots of $\Phi^+$. In this way we are able to meaningfully talk about elements of $\N I$ as being roots.

The set of real roots of $\Phi^+$ is denoted $\Phi^+_{re}$.

For $\nu=\sum_{i\in I}\nu_i\cdot i\in I$, define $|\nu|=\sum_{i\in I }\nu_i$. If $\nu$ happens to be a root, we also call this the height of the root and denote it $\operatorname{ht}(\nu)$.

Let $\Phi_f$ be the underlying finite type root system. A chamber coweight is a fundamental coweight for some choice of positive system on $\Phi_f$. If a positive system is given, let $\Omega$ denote the set of chamber coweights with respect to this system.

Let $p\map{\Phi}{\Phi_{f}}$ denote the projection from the affine root system to the finite root system whose kernel is spanned by the minimal imaginary root $\delta$. For $\a\in\Phi_f$, let $\tilde{\a}$ denote the minimal positive root in $p^{-1}(\a)$.

Let $W=\langle s_i\mid i\in I\rangle$ be the Weyl group of $\Phi$, generated by the simple reflection $s_i$ which is the reflection in the hyperplane perpendicular to $\a_i$.

Let $\D_f$ be the standard set of simple roots in $\Phi_f$. Let $W_f$ be the finite Weyl group.

Let $\P$ denote the set of partitions. A multipartition $\ula=\{\la_\w\}_{\w\in\W}$ is a sequence of partitions indexed by $\W$. We write $\ula\vdash n$ if $\sum_\w |\la_\w|=n$.

The symmetric group on $n$ letters is denoted $S_n$. If $\mu,\nu\in I$, the element $w[\mu,\nu]\in S_{|\mu+\nu|}$ is defined by
\[
w[\mu,\nu](i)=\begin{cases}
i+|\nu| & \text{if $i\leq \mu$} \\
i-|\mu| & \text{otherwise}.
\end{cases}
\]

\section{Convex Orders on Root Systems}

\begin{definition}\label{convexdefn}
A convex order on $\Phi^+$ is a total preorder $\preceq$ on $\Phi^+$ such that
\begin{itemize}
 \item If $\a\preceq \b$ and $\a+\b$ is a root, then $\a\preceq \a+\b\preceq \b$.
\item If $\a\preceq\b$ and $\b\preceq\a$ then $\a$ and $\b$ are imaginary roots.
\end{itemize}
\end{definition}

\begin{theorem}
A convex order $\prec$ on $\Phi^+$ satisfies the following condition:
\begin{itemize}
                                                     \item Suppose $A$ and $B$ are disjoint subsets of $\Phi^+$ such that $\a\prec\b$ for any $\a\in A$ and $\b\in B$. Then the cones formed by the $\R_{\geq 0}$ spans of $A$ and $B$ meet only at the origin.
                                                                             \end{itemize}\end{theorem}
\begin{remark}
In \cite{tingleywebster}, this condition replaces our first condition in their definition of a convex order. This theorem shows that their definition and our definition agree.
\end{remark}
\begin{remark}
 The following proof requires being in finite or affine type since it depends on the positive semidefiniteness of the natural bilinear form.
We do not know if a similar statement is possible for more general root systems.
\end{remark}

\begin{proof}
We will write $(\cdot,\cdot)$ for the natural bilinear form on the root lattice.
Let $\{\a_i\}$ be a finite set of roots in $A$ and let $\{b_j\}$ be a finite set of roots in $B$.
For want of a contradiction, suppose that for some positive real numbers $c_i,d_j$ we have
\begin{equation}\label{cadb}\sum_i c_i\a_i= \sum_j d_j \b_j\end{equation}

Let
\[
 W=\{(x_1,x_2,\ldots,y_1,y_2,\ldots)\mid x_i,y_j\in\Q, \sum_i x_i\a_i=\sum_j y_j\b_j\}.
\]
Then $(c_1,c_2,\ldots,d_1,d_2,\ldots)\in W\otimes_\Q\R$, since the root system $\Phi$ is defined over $\Q$.
As $\Q$ is dense in $\R$, $W$ is dense in $W\otimes_\Q\R$. So there is a point in $W$ with all coordinates positive. Hence we can assume that each $c_i$ and $d_j$ are rational numbers without loss of generality. Clearing denominators, we can assume they lie in $\Z$.


Now suppose we have a solution to (\ref{cadb}) where the $c_i$ and $d_j$ are positive integers with $\sum_i c_i + \sum_j d_j$ as small as possible.

For any $i\neq j$, if $\a_i+\a_j$ were a root, we could replace one occurrence of $\a_i$ and $\a_j$ by the single root $\a_i+\a_j$ to get a smaller solution, contradicting our minimality assumption. Therefore $\a_i+\a_j$ is not a root for any $i\neq j$. This implies that $(\a_i,\a_j)\geq 0$ for $i\neq j$.

If all $\a_i$ and $\b_j$ are imaginary, there is a contradiction since there is only one imaginary direction. So there exists at least one real root in the equation we are studying, without loss of  generality say it is $\a_k$.

Applying $(\cdot,\a_k)$ leaves us with the inequality
\[
 \sum_j d_j (\b_j,\a_k)\geq c_k(\a_k,\a_k)>0.
\]
Therefore there exists $j$ such that $(\b_j,\a_k)>0$, which implies that $\b_j-\a_k$ is a positive root. By convexity this root must be greater than $\b_j$. So now we may subtract $\a_k$ from both sides of (\ref{cadb}) to obtain a smaller solution, again contradicting minimality.

Therefore no solution to (\ref{cadb}) can exist, as required.
\end{proof}

The imaginary roots are all multiples of a fundamental imaginary root, which we
will denote $\d$. In any convex order, these imaginary roots must all be equal to each other.

Let $\prec$ be a convex order. The set of positive real roots is divided into two disjoint subsets, namely
\[
\Phi_{\prec \d} = \{ \a\in \Phi^+ \mid \a\prec \d\},
\]and
\[
\Phi_{\succ \d} = \{ \a \in \Phi^+ \mid  \a\succ \d\}.
\]

If we can write $\Phi_{\prec \d} = \{ \a_1 \prec \a_2\prec \cdots \}$ and $\Phi_{\succ \d} = \{ \b_1 \succ \b_2\succ \cdots \}$ for some sequences of roots $\{\a_i\}_{i=1}^\infty$ and $\{\b_j\}_{j=1}^\infty$, then we say that $\prec$ is of \emph{word type}.

\begin{example}
Let $(V,\leq)$ be a totally ordered $\Q$-vector space. Let $h\map{\Q\Phi}{V}$ be an injective linear transformation. For two positive roots $\a$ and $\b$, say that $\a\prec \b$ if $h(\a)/\operatorname{ht}(\a)<h(\b)/\operatorname{ht}(\b)$ and $\a\preceq \b$ if $h(\a)/\operatorname{ht}(\a)\leq h(\b)/\operatorname{ht}(\b)$. 
This defines a convex order on $\Phi$.
\end{example}

In the above example, we can take $V=\R$ with the standard ordering to get the existence of many convex orders of word type.

An example of a convex order not of word type which we will make use of later on is the following:

\begin{example}\label{tworow}
Let $V=\R^2$ where $(x,y)\leq (x',y')$ if $x<x'$ or $x=x'$ and $y\leq y'$. Let $\D_f$ be a simple system in $\Phi_f$ and pick $\a\in\D_f$. Define $h\map{\Q\Phi}{V}$ by $h(\tilde{\a})=(0,1)$, $h(\tilde{\b})=(x_\b,0)$ for $\b\in\D_f\setminus\{\a\}$ where the $x_\b$ are generically chosen positive real numbers, and $h(\d)=0$. We extend by linearity, noting that $\{\d\}\cup \{\tilde{\b}\mid \b\in\D_f\}$ is a basis of $\Q\Phi$.

In this example, we have
\[
\widetilde{-\a}\prec\widetilde{-\a}+\d+\prec\widetilde{-\a}+2\d\cdots\prec\Z_{>0}\d \prec \cdots \prec \widetilde{\a}+2\d\prec\widetilde{\a}+\d\prec\widetilde{\a}
\] and all other positive roots are either greater than $\widetilde{\a}$ or less than $\widetilde{-\a}$.
\end{example}

Recall that $p$ is the projection from the affine root system to the finite root system.

\begin{lemma}
There exists $w\in W_f$ such that $p(\Phi_{\prec \d})=w \Phi^+_f$ and $p(\Phi_{\succ \d}) = w \Phi^-_f$.
\end{lemma}

\begin{proof}
First suppose that $\a\in p(\Phi_{\prec \d})$ and $-\a\in p(\Phi_{\prec \d})$. Then there are integers $m$ and $n$ such that the affine roots $-\a+m\d$ and $\a+n\d$ are both less than $\d$ in the convex order $\prec$. By convexity, their sum $(m+n)\d$ is also less that $\d$, a contradiction.
Since a similar argument holds for $p(\Phi_{\succ \d})$, we see that for each finite root $\a$, exactly one of $\a$ and $-\a$ lies in $p(\Phi_{\prec \d})$.

Now suppose that $\a,\b\in p(\Phi_{\prec \d})$ and $\a+\b$ is a root. Then for some integers $m$ and $n$, the affine roots $\a+m\d$ and $\b+n\d$ are both less than $\d$. By convexity, their sum $(\a+\b)+(m+n)\d$, which is also an affine root, is also less than $\d$.
Therefore $\a+\b\in p(\Phi_{\prec \d})$.

We have shown that $p(\Phi_{\prec \d})$ is a positive system in the finite root system $\Phi_f$. This suffices to prove the lemma.
\end{proof}

Define a finite initial segment to be a finite set of roots $\a_1\prec \a_2\prec\cdots\prec \a_N$ such that for all positive roots $\b$, either $\b\succ \a_i$ for all $i=1,\ldots,N$ or $\b=\a_i$ for some $i$.

For any $w\in W$ define $\Phi(w)=\{\a\in\Phi^+\mid w\inv\a\in\Phi^-\}$.

\begin{lemma}\label{finiteinitial}
Let $\a_1\prec \a_2 \prec\cdots \prec \a_N$ be a finite initial segment. Then there exists $w\in W $ such that $\{ \a_1,\ldots \a_N \} =\Phi(w)$. 
Furthermore there exists a reduced expression $w=s_{i_1}\cdots s_{i_N}$ such that $\a_k=s_{i_1}\cdots s_{i_{k-1}} \a_{i_k}$ for $k=1,\ldots N$.
\end{lemma}

\begin{proof}
The proof proceeds by induction on $N$. For the base case where $N=1$, any root $\a$ which is not simple is the sum of two roots $\a=\b+\ga$. By convexity of $\prec$, either $\b\prec\a\prec\ga$ or $\ga\prec\a\prec\b$. Either way, $\a\neq\a_1$ so $\a_1$ is simple, $\a_1=\a_i$ for some $i\in I$ and we take $w=s_i$.

Now assume that the result is known for initial segments with fewer than $N$ roots.
Let $v=s_{i_1}\ldots s_{i_{N-1}}$. Consider $v^{-1}\a_N$. By inductive hypothesis, it is a positive root. Suppose for want of a contradiction that $v^{-1}\a_N$ is not simple.
Then we can find positive roots $\b$ and $\ga$ such that $v^{-1}\a_N=\b+\ga$.

We can't have $v\b=\a_N$ as this would force $\ga=0$. If $v\b=\a_j$ for some $j<N$ then $\b=v^{-1}\a_j$ which by inductive hypothesis is in $\Phi^-$, a contradiction. Therefore either $v\b$ is a positive root satisfying $v\b\succ \a_N$ or $v\b\in \Phi^-$. A similar statement holds for $v\ga$.

To have both $v\b$ and $v\ga$ greater than $\a_N$ contradicts the convexity of $\prec$. Therefore, without loss of generality, we may assume $v\b\in\Phi^-$. Then $-v\b$ is a positive root with $v^{-1}(-v\b)=-\b$ which is a negative root, so by inductive hypothesis, $-v\b=\a_j$ for some $j<N$. Now consider the equation $\a_N + (-v\b)=v\ga$. The convexity of $\prec$ implies that $v\ga=\a_{j'}$ for some $j'<N$. This option is shown to be impossible in the previous paragraph, creating a contradiction. Therefore $v^{-1}\a_N$ must be a simple root.

Define $i_N\in I$ by $\a_{i_N}=v^{-1}\a_N$ and let $w=s_{i_1}\cdots s_{i_N}$. It remains to show that $\{ \a_1,\ldots \a_N \} = \{\a\in \Phi^+ \mid w^{-1}\a\in \Phi^-\}$.

If $\b$ is a positive root that is not equal to $\a_j$ for some $j\leq N$, then by inductive hypothesis $v^{-1}\b\in \Phi^+$. Then $w^{-1}\b=s_{i_N}(v^{-1}\b)\in\Phi^-$ if and only if $v^{-1}\b=\a_{i_N}$ which isn't the case since this is equivalent to $\b=\a_N$.

If $\b=\a_j$ for some $j<N$ then $v^{-1}\b\in\Phi^-$. So $w^{-1}\b=s_{i_N}(v^{-1}\b)\in\Phi^-$ unless $v^{-1}\b=-\a_{i_N}$. This isn't the case since it is equivalent to $\b=-\a_{i_N}$.

The above two paragraphs show that for a positive root $\b$, if $\b\in\{\a_1,\ldots ,\a_{N-1}\}$ then $w^{-1}\b\in\Phi^+$ while if $\b\notin \{\a_1,\ldots,\a_N\}$, then $w^{-1}\b\in\Phi^+$. Since $w^{-1}\a_N=-\a_{i_N}\in\Phi^-$, this completes the proof.
\end{proof}

\begin{lemma}\cite{ito}
The restriction of a convex order to $\Phi_{\prec \d}$ is of the form
\[
 \a_{11}\prec\a_{12}\prec\cdots \prec \a_{21}\prec\a_{22}\prec\cdots \,\,\ \cdots\prec \a_{n1}\prec\a_{n2}\prec\cdots
\]
for some positive integer $n$.
\end{lemma}

For a convex order $\prec$, define
\[
I(\prec) = \{\a\in\Phi^+\mid \{\b\in \Phi^+\mid \b\prec \a\} \mbox{ is finite}\}
\]
 \begin{lemma}\label{3.10}
Let $\prec$ be a convex order. Let $\b$ be the smallest root that is not in any initial segment of $\Phi^+$. Assume that $\b$ is real. Let $S$ be a finite set of roots containing $\b$. Then there exists a convex order $\prec'$ such that $I(\prec')=I(\prec)\cup \{\b\}$ and the restrictions of $\prec $ and $\prec'$ to $S$ are the same.
\end{lemma}

\begin{proof}
Let $L$ be the set of roots in $\Phi^+$ less than or equal to $\b$ under $\prec$. Then by \cite[Theorem 3.12]{cellinipapi}, there exists $v,t\in W$ with $t$ a translation
and $L=\cup_{n=1}^\infty \Phi(vt^n)$.

Let $w$ be such that $S\subset \{\a_1\prec\cdots\prec\a_N\}=\Phi_w$. There exists an integer $n$ such that $\Phi(w)\cup \{\beta\}\subset \Phi(vt^n)$. Let $v'=vt^n$. Since $\Phi(v')\supset \Phi(w)$, for any reduced expression of $w$, there exists a reduced expression of $v'$ beginning with that of $w$.

We choose the reduced decomposition of $w$ to be compatible with $\prec$. Then extend the reduced decomposition as per the above to get a new ordering $\prec'$ on $L$. This has the desired properties.
\end{proof}

\begin{theorem}\label{3.11}
 Let $S$ be a finite subset of $\Phi^+$ and let $\prec$ be a convex order on $\Phi^+$. Then there exists a convex order $\prec'$ of word type such that the restrictions of $\prec$ and $\prec'$ to $S$ are equivalent.
\end{theorem}

\begin{proof}
Suppose our convex order begins 
\[
\a_1\prec \a_2 \prec \cdots \prec \b_1 \prec \b_2 \prec \cdots
\]
and that $S\cap \{\a_i\mid i\in \Z^+\} \subset \{\a_1,\ldots,\a_n\}$. We now define inductively a sequence of convex orders $\prec_i$ with $I(\prec_i)=\{\a_i\mid i\in\Z^+\}\cup \{\b_1,\ldots,\b_i\}$ as follows:

Set $\prec_0=\prec$. Assume that $\prec_{i}$ is constructed. To construct $\prec_{i+1}$, apply Lemma \ref{3.10} with $S=\{\a_1,\ldots,\a_{n+i},\b_1,\ldots,\b_{i}\}$. We will take the convex order denoted $\prec'$ whose existence is given to us by Lemma \ref{3.10} as $\prec_{i+1}$.

Now let $\prec''=\lim_{i\to \infty} \prec_i$. If $\prec$ is of $n$-row type, then $\prec''$ will be of $(n-1)$-row type and the restrictions of $\prec$ and $\prec''$ to $S$ are the same. After iterating this process we reach a new convex order $\prec'$ whose restriction to $S$ is the same as $\prec$ and is of word type on $\Phi_{\prec\d}$. Repeating this construction on the set of roots greater than $\d$ completes the proof of this theorem.
\end{proof}

\begin{remark}
Using this theorem it will often be possible to assume without loss of generality that the convex order $\prec$ is of word type.
\end{remark}



\section{The Algebra $\f$}\label{eff}

The algebra $\f_{\Q(q)}$ is the $\Q(q)$ algebra as defined in \cite{lusztigbook} generated by elements $\{\th_i\mid i\in I\}$. Lusztig defines it as the quotient of a free algebra by the radical of a bilinear form. By the quantum Gabber-Kac theorem, it can also be defined in terms of the Serre relations.
Morally, $\f_{\Q(q)}$ should be thought of as the positive part of the quantised enveloping algebra $U_q(\g)$.
There is only a slight difference in the coproduct, necessary as the coproduct in $U_q(\g)$ does not map $U_q(\g)^+$ into $U_q(\g)^+\otimes U_q(\g)^+$.

There is a $\Z[q,q\inv]$-form of $\f_{\Q(q)}$, which we denote simply by $\f$. It is the $\Z[q,q\inv]$-subalgebra of $\f_{\Q(q)}$ generated by the divided powers $\th_i^{(n)}=\th_i^n/[n]_q!$, where $[n]_q!=\prod_{i=1}^n(q^i-q^{-i})/(q-q\inv)$ is the $q$-factorial. If $\A$ is any $\Z[q,q\inv]$-algebra, we use the notation $\f_\A$ for $\A\otimes_{\Z[q,q\inv]}\f$.

The algebra $\f$ is graded by $\N I$ where $\th_i$ has degree $i$ for all $i\in I$.
We write $\f=\oplus_{\nu\in\N I} \f_\nu$ for its decomposition into graded components.
Of significant importance for us is the dimension formula
\begin{equation}\label{dimfnu}
\sum_{\nu\in \N I}\dim \f_\nu t^\nu = \prod_{\a\in\Phi^+} (1-t^{\a})^{-\operatorname{mult}(\a)}
\end{equation}

The tensor product $\f\otimes \f$ has an algebra structure given by
\[
 (x_1\otimes y_1)(x_2\otimes y_2)=q^{\b_1\cdot\a_2}x_1 x_2 \otimes y_1 y_2
\] where $y_1$ and $x_2$ are homogeneous of degree $\b_1$ and $\a_2$ respectively.

Given a bilinear form $(\cdot,\cdot)$ on $\f$, we obtain a bilinear form $(\cdot,\cdot)$ on $\f\otimes \f$ by
\[
 (x_1 \otimes x_2,y_1\otimes y_2)=(x_1,x_2)(y_1,y_2).
\]

There is a unique algebra homomorphism $r\map{\f}{\f\otimes \f}$ such that $r(\th_i)=\th_i\otimes 1 + 1 \otimes \th_i$ for all $i\in I$.

The algebra $\f$ has a symmetric bilinear form $\langle\cdot,\cdot\rangle$ satisfying
\begin{eqnarray*}
 \langle\th_i,\th_i\ra&=&(1-q^2)^{-1} \\
\langle x,yz\ra &=& \langle r(x),y\otimes z\ra.
\end{eqnarray*}

The form $\langle \cdot,\cdot\ra$ is nondegenerate. Indeed, in the definition of $\f$ in \cite{lusztigbook}, $\f$ is defined to be the quotient of a free algebra by the radical of this bilinear form. It is known that $\f$ is a free $\Z[q,q\inv]$-module. Let $\f^*$ be the graded dual of $\f$ with respect to $\langle \cdot,\cdot\ra$. By definition, $\f^*=\oplus_{\nu\in\N I} \f_\nu^*$. As twisted bialgebras over $\Q(q)$, $\f_{\Q(q)}$ and $\f^*_{\Q(q)}$ are isomorphic, though there is no such isomorphism between their integral forms.

\section{KLR Algebras}\label{klr}

A good introduction to the basic theory of KLR algebras appears in \cite[\S 4]{kleschevram}.
Although it is not customary, we will first give the geometric construction of KLR algebras, then discuss the standard presentation in terms of generators and relations.
In this paper, we must restrict ourselves to KLR algebras which come from geometry. The primary reason for this restriction is our reliance on the theory of $R$-matrices, which we introduce in $\S\ref{sim}$. The results presented in $\S\ref{ext}$ also require the geometric interpretation.

Define a graph with vertex set $I$ and with $-i\cdot j$ edges between $i$ and $j$ for all $i\neq j$. Let $Q$ be the quiver obtained by placing an orientation on this graph.

For $\nu\in\N I$, define $E_\nu$ and $G_\nu$ by $$E_\nu=\prod_{i\to j}\Hom_{\mathbb{C}} (\mathbb{C}^{\nu_i},\mathbb{C}^{\nu_j}),$$
$$G_\nu=\prod_i GL_{\nu_i}(\mathbb{C}).$$

With the obvious action of $G_\nu$ on $E_\nu$, $E_\nu/G_\nu$ is the moduli stack of representations of $Q$ with dimension vector $\nu$.

Let $\flags_\nu$ be the complex variety whose points consist of a point of $E_\nu$, together with a full flag of subrepresentations of the corresponding representation of $Q$. 
The variety $\flags_\nu$ is a disjoint union of smooth connected varieties. Let $\flags_\nu=\sqcup_\ii \flags_\nu^\ii$ be its decomposition into connected components and $\pi^\ii\map{\flags_\nu^\ii}{E_\nu}$ be the natural $G_\nu$-equivariant morphism. Define
\[
\mathcal{L}=\bigoplus_\ii (\pi^\ii)_! \underline{\Q}_{\flags_\nu^\ii}[\dim \flags_\nu^\ii] \in D^b_{G_\nu}(E_\nu).
\]

For each $\nu\in\N I$ we define the KLR algebra $R(\nu)$ by
\[
 R(\nu)=\bigoplus_{d\in\Z} \Hom_{D^b_{G_\nu}(E_\nu)}(\mathcal{L},\mathcal{L}[d]).
\]

We now introduce the more customary approach via generators and relations. This presentation is due to \cite{vv} and \cite{rouquier}, and more recently over $\Z$ in \cite{maksimau}.
To introduce this presentation, we first need to define, for any $\nu\in\N I$,
\[
 \Seq(\nu)=\{\ii=(\ii_1,\ldots,\ii_{|\nu|})\in I^{|\nu|}\mid\sum_{j=1}^{|\nu|} \ii_j=\nu\}.
\]
This is acted upon by the symmetric group $S_{|\nu|}$ in which the adjacent transposition $(i,i+1)$ is denoted $s_i$.

Define the polynomials $Q_{i,j}(u,v)$ for $i,j\in I$ by
\[
 Q_{i,j}(u,v)=\begin{cases}
\prod_{i\to j}(u-v) \prod_{j\to i}(v-u)
               &\text{if } i\neq j \\
0 &\text{if } i=j 
              \end{cases}
\] where the products are over the sets of edges from $i$ to $j$ and from $j$ to $i$, respectively.

\begin{theorem}
The KLR algebra $R(\nu)$ is the associative $\Q$-algebra generated by elements $e_{\bf i}$, $y_j$, $\tau_k$ with ${\bf i}\in \seq(\nu)$, $1\leq j\leq |\nu|$ and $1\leq k< |\nu|$ subject to the relations
\begin{equation} \label{eq:KLR}
\begin{aligned}
& e_\ii e_{\uj} = \delta_{\ii, \uj} e_\ii, \ \
\sum_{\ii \in \Seq(\nu)}  e_\ii = 1, \\
& y_{k} y_{l} = y_{l} y_{k}, \ \ y_{k} e_\ii = e_\ii y_{k}, \\
& \tau_{l} e_\ii = e_{s_{l}\ii} \tau_{l}, \ \ \tau_{k} \phi_{l} =
\tau_{l} \phi_{k} \ \ \text{if} \ |k-l|>1, \\
& \tau_{k}^2 e_\ii = Q_{\ii_{k}, \ii_{k+1}} (y_{k}, y_{k+1})e_\ii, \\
& (\tau_{k} y_{l} - y_{s_k(l)} \tau_{k}) e_\ii = \begin{cases}
-e_\ii \ \ & \text{if} \ l=k, \ii_{k} = \ii_{k+1}, \\
e_\ii \ \ & \text{if} \ l=k+1, \ii_{k}=\ii_{k+1}, \\
0 \ \ & \text{otherwise},
\end{cases} \\[.5ex]
& (\tau_{k+1} \tau_{k} \tau_{k+1}-\tau_{k} \tau_{k+1} \tau_{k}) e_\ii\\
& =\begin{cases} \dfrac{Q_{\ii_{k}, \ii_{k+1}}(y_{k},
y_{k+1}) - Q_{\ii_{k}, \ii_{k+1}}(y_{k+2}, y_{k+1})}
{y_{k} - y_{k+2}}e_\ii\ \ & \text{if} \
\ii_{k} = \ii_{k+2}, \\
0 \ \ & \text{otherwise}.
\end{cases}
\end{aligned}
\end{equation}
\end{theorem}

\begin{remark}
Although the polynomials $Q_{i,j}(u,v)$ are not exactly as they appear in \cite{khovanovlauda}, the reader should not be concerned when we quote results from \cite{khovanovlauda} as all of the arguments go through without change.
The discussion in \cite{kl2} shows that changing the ordering of the quiver $Q$ does not change the isomorphism type of $R(\nu)$.
\end{remark}



The KLR algebras $R(\nu)$ are $\Z$-graded, where $e_{\ii}$ is of degree zero, $y_j e_\ii$ is of degree $\ii_j \cdot \ii_j$ and $\phi_k e_\ii$ is of degree $-\ii_k \cdot \ii_{k+1}$.

They satisfy the property that $R(\nu)_d=0$ for $d$ sufficiently negative (depending on $\nu$) and $R(\nu)_d$ is finite dimensional for all $d$. Relevant implications of these properties are that there are a finite number of isomorphism classes of simple modules and that projective covers exist.

All representations of KLR algebras that we consider will be finitely generated $\Z$-graded representations. If needed, we write $M=\oplus_d M_d$ for the decomposition of a module $M$ into graded pieces. A submodule of a finitely generated $R(\nu)$-module is finitely generated by \cite[Corollary 2.11]{khovanovlauda}.

For a module $M$, we denote its grading shift by $i$ by $q^iM$, this is the module with $(q^i M)_d = M_{d- i}$.

Given two modules $M$ and $N$, we consider $\Hom(M,N)$, and more generally $\Ext^i(M,N)$ as graded vector spaces. All Ext groups which appear in the paper will be taken in the category of $R(\nu)$-modules.

Let $\tau$ be the antiautomorphism of $R(\nu)$ which is the identity on all generators $e_\ii,y_i,\phi_j$. For any $R(\nu)$-module $M$, there is a dual module $M\dual = \Hom_\Q (M,\Q)$, where the $R(\nu)$ action is given by $(x\la)(m)=\la(\tau(x)m)$ for all $x\in R(\nu)$, $\la\in M\dual$ and $m\in M$.

For every irreducible $R(\nu)$-module $L$, there is a unique choice of grading shift such that $L\dual\cong L$. \cite{khovanovlauda}

Let $\la,\mu\in\N I$. 
Then there is a natural inclusion $\iota_{\la,\mu}:R(\la)\otimes R(\mu)\to R(\la+\mu)$, defined by $\iota_{\la,\mu}(e_\ii\otimes e_\jj)=e_{\ii\jj}$, $\iota_{\la,\mu}(y_i\otimes 1)=y_i$, $\iota_{\la,\mu}(1\otimes y_i)=y_{i+|\la|}$, $\iota_{\la,\mu}(\phi_i\otimes 1)=\phi_i$, $\iota_{\la,\mu}(1\otimes \phi_i)=\phi_{i+|\la|}$.

Define the induction functor $\Ind_{\la,\mu}:R(\la)\otimes R(\mu)\mods \to R(\la+\mu)\mods$ by
\[
 \Ind_{\la,\mu} (M) = R(\la+\mu)\bigotimes_{R(\la)\otimes R(\mu)} M.
\]
Define the restriction functor $\Res_{\la,\mu}:R(\la+\mu)\mods \to R(\la)\otimes R(\mu)\mods$ by
\[
 \Res_{\la,\mu}(M) = \iota_{\la,\mu}(1_{R(\la)\otimes R(\mu)})M.
\]
 The induction and restriction functors are both exact.

For a $R(\la)$-module $A$ and a $R(\mu)$-module $B$, we write $A\circ B$ for $\Ind_{\la,\mu}(A\otimes B)$. Under duality, the behaviour is
\begin{equation}\label{dualofaproduct}
(A\circ B)\dual \cong q^{\la\cdot\mu} B\dual \circ A\dual.
\end{equation}

Khovanov and Lauda \cite{khovanovlauda} prove the existence of a dual pair of isomorphisms
\begin{equation}\label{pmod}
 \bigoplus_{\nu\in \N I} G_0(R(\nu)\prmod) \cong \f
\end{equation}
and
\begin{equation}\label{fmod}
 \bigoplus_{\nu\in \N I} K_0(R(\nu)\fmod) \cong \f^*.
\end{equation}

The category $R(\nu)\prmod$ is the category of finitely generated projective $R(\nu)$-modules and $G_0$ means to take the split Grothendieck group. The category $R(\nu)\fmod$ is the category of finite dimensional $R(\nu)$-modules and $K_0$ means to take the Grothendieck group. We denote the class of a module $M$, identified with its image under the above isomorphisms, by $[M]$.
The action of $q\in \A$ is by grading shift.

The functors of induction and restriction decategorify to a product and coproduct. The isomorphisms above are then isomorphisms of twisted bialgebras.

If $M$ is a general finitely generated $R(\nu)$-module, then it has a well-defined composition series, where each composition factor appears with a multiplicity that is an element of $\Z((q))$. Thus we can consider $[M]$ to be an element of $\f^*_{\Z((q))}$.

Of great importance will be the following Mackey theorem. The general case stated below has the same proof as the special case presented in \cite{khovanovlauda}.

\begin{theorem}\cite[Proposition 2.18]{khovanovlauda}\label{mackey}
Let $\la_1,\ldots,\la_k,\mu_1\ldots,\mu_l\in\N I$ be such that $\sum_i \la_i=\sum_j \mu_j$ and let $M$ be a $R(\la_1)\otimes \cdots \otimes R(\la_k)$-module. Then the module
$ \Res_{\mu_1,\ldots,\mu_l}\circ\Ind_{\la_1,\ldots,\la_k}(M)
$ has a filtration indexed by tuples $\nu_{ij}$ satisfying $\la_i=\sum_j \nu_{ij}$ and $\mu_j=\sum_i\nu_{ij}$.
The subquotients of this filtration are isomorphic, up to a grading shift, to the composition $\Ind_\nu^\mu\circ \tau\circ \Res_\nu^\la(M)$. Here
$\Res_{\nu}^\la\map{\otimes_i R(\la_i)\mods}{\otimes_i(\otimes_jR(\nu_{ij}))\mods}$ is the tensor product of the $\Res_{\nu_{i\bullet}}$, $\tau\map{\otimes_i(\otimes_jR(\nu_{ij}))\mods}{\otimes_j(\otimes_iR(\nu_{ij}))\mods}$ is given by permuting the tensor factors and $\Ind_\nu^\mu\map{\otimes_j(\otimes_iR(\nu_{ij}))\mods}{\otimes_j R(\mu_j)\mods}$ is the tensor product of the $\Ind_{\nu_{\bullet i}}$.
\end{theorem}

We refer to the filtration appearing in the above theorem as the \emph{Mackey filtration}. It will be very common for us to make arguments using vanishing properties of modules under restriction to greatly restrict the number of these subquotients which can be nonzero.

For each $w\in S_{|\nu|}$, make a choice of a reduced decomposition $w=s_1\ldots s_n$ as a product of simple reflections. Define $\tau_w=\tau_{1}\cdots \tau_{n}$. In general $\tau_w$ depends on the choice of reduced decomposition though this is not the case for permutations of the form $w[\b,\ga]$.

\begin{theorem}\cite[Theorem 2.5]{khovanovlauda}\cite[Theorem 3.7]{rouquier}\label{basisofrnu}	
 The set of elements of the form $y_1^{a_1}\cdots y_{|\nu|}^{a_{|\nu|}}\tau_w e_\ii$ with $a_1,\ldots,a_{|\nu|}\in\N$, $w\in S_{|\nu|}$ and $\ii\in\Seq(\nu)$ is a basis of $R(\nu)$.
\end{theorem}

\begin{theorem}\cite[Corollary 2.11]{khovanovlauda}
The KLR algebra is Noetherian.
\end{theorem}

We work over the ground field $\Q$. It is proved in \cite{khovanovlauda} that any irreducible module is absolutely irreducible, so there is no change to the theory in passing to a field extension. 
This also means that any irreducible module for a tensor product of KLR algebras is a tensor product of irreducibles, a fact we use without comment.

\section{Adjunctions}
In addition to the induction and restriction functor defined in the previous section, there is also a coinduction functor $\Coind_{\la,\mu}R(\la)\otimes R(\mu)\mods \to R(\la+\mu)\mods $, defined by
\[
 \Coind_{\la,\mu}(M) = \Hom_{R(\la)\otimes R(\mu)} (R(\la+\mu),M) 
\]
where the $R(\la+\mu)$ module structure on $\Coind_{\la,\mu}(M)$ is given by $(rf)(t)=f(tr)$ for $f\in\Coind_{\la,\mu}(M)$ and $r,t\in R(\la+\mu)$.

The following adjunctions are standard:

\begin{proposition}
The functor $\Ind_{\la,\mu}$ is left adjoint to $\Res_{\la,\mu}$, while the functor $\Coind_{\la,\mu}$ is right adjoint to $\Res_{\la,\mu}$.
\end{proposition}

As a $R(\la)\otimes R(\mu)$ module, $R(\la+\mu)$ is free of finite rank. This implies that the induction, restriction and coinduction functors all send projective modules to projective modules. As a consequence, there are natural isomorphisms of higher Ext groups
\begin{equation}\label{extind}
 \Ext^i(A\circ B, C)\cong \Ext^i(A\otimes B,\Res_{\la,\mu} C)
\end{equation}
for all $A\in R(\la)\mods$, $B\in R(\mu)\mods$ and $C\in R(\la+\mu)\mods$.

Let $\sigma_\nu\map{R(\nu)}{R(\nu)}$ be the involutive isomorphism of $R(\nu)$ with $\sigma_\nu(e_\ii)=e_{w_0\ii}$, $\sigma_\nu(y_i)=y_{|\nu|+1-i}$ and $\sigma_\nu(\tau_j e_\ii)=(1-2\delta_{\ii_j,\ii_{j+1}})\tau_{|\nu|-j}e_{w_0\ii}$. This induces an autoequivalence $\sigma_\nu^*$ of $R(\nu)\mods$.

\begin{theorem}\cite[Theorem 2.2]{laudavazirani} 
There is a natural equivalence of functors
\[
 \sigma_{\la+\mu}^* \circ \Ind_{\la,\mu} \cong q^{(\la\cdot\mu)} \Coind_{\la,\mu} \circ (\sigma_\la^*\otimes\sigma_\mu^*).
\]
\end{theorem}

\begin{proof}
 The statement of this theorem in \cite{laudavazirani} includes a hypothesis that the modules in question are all finite dimensional. 
 Exactly the same proof works for graded modules all of whose pieces are finite dimensional, which covers all the modules we will ever come across. The general case follows by writing a module as the direct limit of its finitely generated submodules (noting that $R(\la+\mu)$ is finite over $R(\la)\otimes R(\mu)$).
\end{proof}

\begin{remark}
 Most importantly, applied to a module of the form $A\otimes B$ yields an isomorphism
\begin{equation*}
 \Ind_{\la,\mu}(A\otimes B) \cong q^{(\la\cdot\mu)} \Coind_{\mu,\la} (B\otimes A).
\end{equation*}
In particular, there is an isomorphism
\begin{equation}\label{oppositeadjunction}
 \Ext^i(A,B\circ C)\cong q^{-(\la\cdot\mu)}\Ext^i( \Res_{\la,\mu} A, C\otimes B)
\end{equation}
for any $R(\la)$-module $C$, $R(\mu)$-module $B$ and $R(\la+\mu)$-module $A$.
\end{remark}

There are parabolic analogues of all the functors and results discussed in this section.

\section{The Ext Bilinear Form}\label{ext}

By the decomposition theorem \cite{bbd}, we have
\[
 \mathcal{L}\cong\bigoplus_{b\in\B_\nu} L_b\otimes \P_b
\]
where each $L_b$ is a nonzero finite dimensional graded vector space and $\P_b$ is an 
irreducible $G_\nu$-equivariant perverse sheaf on $E_\nu$. The indexing set $\B_\nu$ can be taken to be the set of elements of weight $\nu$ in the crystal $\mathcal{\B}(\infty)$, though for our purposes it is not necessary to know this fact.

The maximal semisimple quotient of $R(\nu)$ is $\oplus_{b\in \B_\nu} \End(L_b)$ and hence the simple representations of $R(\nu)$ are the multiplicity spaces $L_b$. 
The projective cover of $L_b$ is the module $\oplus_{d\in \Z}\Hom_{D^b_{G_\nu}(E_\nu)}(\mathcal{L},\P_b[d])$.
In this way we get a bijection between simple perverse summands of
$\pi_!\underline{\Q}$ and irreducible representations of $R(\nu)$. By Lusztig's geometric construction of canonical bases, the class of a simple representation under the isomorphism (\ref{fmod}) lies in the dual canonical basis while the class of its projective cover under (\ref{pmod}) lies in the canonical basis.

As has been noted by Kato \cite{kato}, each algebra $R(\nu)$ is graded Morita equivalent to the algebra
\[
A(\nu)=\bigoplus_{d\in\Z}\Hom_{D^b_{G_\nu}(E_\nu)}\left(\bigoplus_{b\in\B_\nu} \P_b,\bigoplus_{b\in\B_\nu} \P_b[d]\right).
\]The algebra $A(\nu)$ is a $\N$-graded algebra with $A(\nu)_0$ semisimple. Under this Morita equivalence the self-dual irreducible module $L_b$ gets sent to a one-dimensional representation of $A(\nu)$ concentrated in degree zero.

\begin{lemma}\label{extwelldefined}
Let $M$ be a finitely generated representation of $R(\nu)$ and let $N$ be a finite dimensional representation of $R(\nu)$. Fix an integer $d$.
Then there exists $i_0$ such that $\Ext^i(M,N)_d=0$ for all $i>i_0$.
\end{lemma}

\begin{proof}
 Replace $R(\nu)$ with the Morita equivalent algebra $A(\nu)$ and assume that $M$ and $N$ are $A(\nu)$-modules. Let $\cdots \to P^1\to P^0 \to M\to 0$ be a minimal projective resolution of $M$. As $M$ is finitely generated, there exists $d_0$ such that $M_j=0$ for $j<d_0$. Since $A(\nu)$ is nonnegatively graded with $A(\nu)_0$ semisimple, $P^i_j=0$ for $j<d_0+i$. The vector space $\Ext^i(M,N)$ is a subquotient of $\Hom(P^i,N)$ and for sufficiently large $i$, $\Hom(P^i,N)_d=0$ by degree considerations.
\end{proof}

By the above lemma, if $M$ is a finitely generated $R(\nu)$-module and $N$ is a finite dimensional $R(\nu)$-module, then the infinite sum
\[
  ( M,N ) = \sum_{i=0}^\infty (-1)^i \dimq \Ext^i(M,N).
\]
is a well-defined element of $\Z((q))$. We thus get a pairing on Grothendieck groups
\[
 ( \cdot,\cdot )\map{\f^*_{\Z((q))}\times\f^*}{\Z((q))}.
\]

\begin{lemma}\label{extprops}
 The pairing $(\cdot,\cdot)$ satisfies the following properties
\begin{align*}
 ( f(q)x,g(q) y) &= f(q)g(q\inv)( x,y ) \\
( \th_i,\th_i^* ) &= 1 \\
( xy,z) &=( x\otimes y,r(z)) \\
( x,yz) &= q^{\b\cdot\ga}( r(x),z\otimes y)
\end{align*}
for all $x,y,z\in\f$, $f(q)\in\Z((q))$ ,$\g(q)\in\Z[q,q\inv]$, 
where $y$ and $z$ are homogeneous of degree $\b$ and $\ga$.
\end{lemma}

\begin{proof}
 The first formula is obvious. The second is a simple computation in $R(i)\cong k[z]$. The third follows from (\ref{extind}) and the fourth follows from (\ref{oppositeadjunction}).
\end{proof}

Let $\langle x,y\ra=( x,\bar{y})$. The pairing $\langle \cdot,\cdot\ra$ can be extended by $\Z((q))$-linearity to give a bilinear pairing on $\f^*_{\Z((q))}$.

\begin{lemma}\label{extprops2}
 The pairing $\langle \cdot,\cdot\ra$ satisfies the following properties
\begin{align*}
 \langle f(q)x,g(q)y\ra &=f(q)g(q)\langle x,y\ra \\
\langle\th_i,\th_i\ra &=(1-q^2)\inv \\
\langle xy,z\ra &=\langle x\otimes y,r(z)\ra \\
\langle x,yz\ra &=\langle r(x),y\otimes z\ra
\end{align*}
\end{lemma}

\begin{proof}
 These follow from the analogous formulae in Lemma \ref{extprops}. To derive the third we need to know that $r$ commutes with the bar involution while to derive the fourth we need to know that $\overline{yz}=q^{\b\cdot\ga}\bar{z}\bar{y}$ for homogeneous elements $y$ and $z$ of degree $\b$ and $\ga$.
\end{proof}

\begin{corollary}
 The pairing $\langle \cdot,\cdot\ra $ defined using the Ext-pairing is equal to the usual pairing on $\f$ as in the end of $\S\ref{eff}$.
\end{corollary}

\begin{proof}
 It is immediate that there is a unique pairing satisfying the properties of Lemma \ref{extprops2} and these properties define the pairing in \cite{lusztigbook}.
\end{proof}

\begin{lemma}\label{geometry}
Let $M$ be a finite dimensional $R(\nu)$-module with
\[
[M]=\sum_{i=m}^n \sum_L a_{i,L} q^i [L]
\]
where the second sum is over all self-dual simple modules $L$.
If $a_{n,L}\neq 0$ then $q^n L$ is a submodule of $M$ while if $a_{m,L}\neq 0$ then $q^m L$ is a quotient of $M$.
\end{lemma}

\begin{proof}
 If this lemma is false, then there exist self-dual irreducible representations $L_1$ and $L_2$ of $R(\nu)$, and an integer $d\leq 0$ such that $\Ext^1(L_1,L_2)_d\neq 0$. Now replace $R(\nu)$ by the Morita equivalent $A(\nu)$. We compute $\Ext^1(L_1,L_2)$ by computing a minimal projective resolution of $L_1$. Since $A(\nu)$ is non-negatively graded with $A(\nu)_0$ semisimple, we see from this computation that $\Ext^1(L_1,L_2)$ is concentrated in degrees greater than zero.
\end{proof}

\section{Proper Standard Modules}\label{properstandard}

\begin{definition}
 Let $\a$ be a positive root and $n$ be an integer. A representation $L$ of $R(n\a)$ is called semicuspidal if
$\Res_{\la,\mu}L\neq 0$ implies that $\la$ is a sum of roots less than or equal to $\a$ and $\mu$ is a sum of roots greater than or equal to $\a$.
\end{definition}

\begin{lemma}\label{10.2}
Let $\a$ be a positive root, $m_1,\ldots,m_n\in\Z^+$ and $L_i$ be a semicuspidal representation of $R(m_i\a)$ for each $i=1,2,\ldots,n$. Then the module $L_1\circ\cdots\circ L_n$ is semicuspidal.
\end{lemma}

\begin{proof}
This immediate from Theorem \ref{mackey} and the definition of semicuspidality.
\end{proof}

\begin{definition}\label{cuspidaldefn}
 Let $\a$ be a positive root. A representation $L$ of $R(\a)$ is called cuspidal if whenever $\Res_{\la,\mu}L\neq 0$ and $\la,\mu\neq 0$, we have that $\la$ is a sum of roots less than $\a$ and $\mu$ is a sum of roots greater than $\a$.
\end{definition}

\begin{remark}It is clear that if $\a$ is an indivisible root then any semicuspidal representation of $R(\a)$ is cuspidal. If $\a=n\d$ for $n\geq 2$ we will see in Theorem \ref{19.7} that there are no cuspidal representations of $R(\a)$.
\end{remark}

\begin{definition}
A sequence of modules $L_1,\ldots,L_n$ is called \emph{admissible} if each $L_i$ is an irreducible semicuspidal representation of $R(m_i\a_i)$ with $m_i\in\Z^+$ and the positive roots $\a_i$ satisfy $\a_1\succ\a_2\succ\cdots\succ\a_n$.
\end{definition}

\begin{lemma}\label{restrictinduct}
Let $\a_1\succ\a_2\succ\cdots \succ \a_k$ and $\b_1\succ\b_2\succ\cdots\succ \b_l$ be positive roots and $m_1,\ldots,m_k$, $n_1,\ldots,n_l$ be positive integers. Let $L_1,\ldots,L_k$ be semicuspidal representations of $R(m_1\a_1),\cdots,R(m_k\a_k)$ respectively. Then
\[
\Res_{n_1\b_1,\ldots,n_l\b_l} L_1\circ \cdots \circ L_k = \begin{cases}
 0 &\mbox{unless } \underline{\b}\leq\underline{\a} \\
L_1\otimes \cdots\otimes L_k &\mbox{if } \underline{\b}=\underline{\a},
\end{cases}
\] where we are considering bilexicographical ordering on the multisets $\underline{\a}$ and $\underline{\b}$.
\end{lemma}

\begin{proof}
Consider a nonzero layer of the Mackey filtration for $\Res_{n_1\b_1,\ldots,n_l\b_l} L_1\circ\cdots\circ L_n $. It is indexed by a set of elements $\nu_{ij}\in \N I$ such that $m_i\a_i=\sum_j \nu_{ij}$ and $n_j\b_j = \sum_{i}\nu_{ij}$. For the piece of the filtration to be nonzero, it must be that $\Res_{\nu_{i,1},\ldots,\nu_{i,n}}L_i \neq 0$ for each $i$.

Suppose that $t$ is an index such that $m_i\a_i=n_i\b_i$ for $i<t$. We will prove that in order for us to have a nonzero piece of the filtration, it must be that either $\b_t\prec \a_t$ or $m_t\a_t=n_t\b_t=\nu_{t,t}$.

By induction on $t$, we may assume that $\nu_{ii}=m_i\a_i=n_i\b_i$ for $i<t$. Therefore $\nu_{i,j}=0$ for all $i$ and $j$ with $i\geq t$ and $j<t$.

Suppose $i\geq t$. Since the module $L_i$ is cuspidal, this implies that $\nu_{i,t}$ is a sum of roots less than or equal to $\a_i$, which are all less than or equal to $\a_t$.

Now $n_t\b_t=\sum_{i\geq t}\nu_{i,t}$ is written as a sum of positive roots all less than or equal to $\a_t$. Therefore, by convexity of the ordering, either $\b_t\prec \a_t$ or $n_t\b_t=m_t\a_t$. In this latter case, equality in our inequalities must hold everywhere, hence $\nu_{t,t}=n_t\b_t$ as required.

This is enough to conclude that $\underline{\a}\geq \underline{\b}$ under lexicographical ordering.
Similarly we get $\underline{\a}\geq \underline{\b}$ under reverse lexicographical ordering, so we have $\underline{\a}\geq \underline{\b}$ under the bilexicographical order.
\end{proof}

\begin{lemma}\label{irreduciblehead}
 Let $\a_1\succ \a_2\succ\cdots\succ \a_n$ be roots, $m_1,\ldots,m_n$ be positive integers and $L_1,\ldots,L_n$ be irreducible semicuspidal representations of $R(m_1\a_1),\ldots,R(m_n\a_n)$ respectively.
Then \begin{enumerate}
\item the module $L_1\circ \cdots \circ L_n$ has a unique irreducible quotient $L$, and
\item $\Res_{m_1\a_1,\ldots,m_n\a_n} L_1\circ\cdots\circ L_n = \Res_{m_1\a_1,\ldots,m_n\a_n}L=L_1\otimes\cdots\otimes L_n$.
\end{enumerate}
\end{lemma}

\begin{proof}
Suppose that $Q$ is a nonzero quotient of $L_1\circ\cdots\circ L_n$. Then by adjunction there is a nonzero map from $L_1\otimes\cdots\otimes L_n$ to $\Res Q$. As $L_1\otimes \cdots\otimes L_n$ is irreducible, this map is injective.

The restriction functor is exact and by Lemma \ref{restrictinduct}, $\Res (L_1\circ\cdots\circ L_n)$ is simple. Therefore the head of $L_1\circ \cdots \circ L_n$ must be simple.
\end{proof}

If $L_1,\ldots,L_n$ is a sequence of representations, we define $A(L_1,\ldots,L_n)=\operatorname{cosoc}(L_1\circ\cdots L_n)$.
The below two theorems appear also in \cite{kleshchev} and \cite{tingleywebster}. We provide proofs of both after the statement of Theorem \ref{numberofimaginarysemicuspidals}.

\begin{theorem}\label{fdmain}
Every irreducible module for $R(\nu)$ is of the form $A(L_1,\ldots,L_n)$ for exactly one set of irreducible semicuspidal representations $L_1,\ldots,L_n$ of $R(m_1\a_1),\ldots,R(m_n\a_n)$ respectively, where $\a_1\succ\cdots\succ\a_n$ are positive roots.
\end{theorem}


\begin{theorem}\label{numberofimaginarysemicuspidals}
If $\a$ is a positive real root and $n$ is a positive integer, there is one simple semicuspidal module for $R(n\a)$. For the imaginary roots, let $f(n)$ be the number of simple semicuspidal representations of $R(n\d)$ (and set $f(0)=1$). Then
\[
\sum_{n=0}^{\infty} f(n) t^n = \prod_{i=1}^\infty (1-t^i)^{1-|I|}.
\]
\end{theorem}

\begin{proof}
Here we prove Theorems \ref{fdmain}
and \ref{numberofimaginarysemicuspidals} by a simultaneous induction on $\nu$.

First let us consider the case where $\nu$ is not of the form $n\a$ for some root $\a$. The number of irreducible representations of $R(\nu)$ is equal to $\dim{\f_\nu}$, which is the coefficient of $t^\nu$ in the power series (\ref{dimfnu}).

By inductive hypothesis applied to Theorem \ref{numberofimaginarysemicuspidals}, the number of admissible sequences of semicuspidal modules $(L_1,\ldots,L_n)$ is equal to $\dim\f_\nu$. By Lemma \ref{irreduciblehead}, each of the modules $A(L_1,\ldots,L_n)$ are irreducible, and by applying various restriction functors, we see via Lemma \ref{restrictinduct} that these modules are all distinct. Therefore we have identified all of the irreducible $R(\nu)$-modules in this case, proving Theorem \ref{fdmain}.

Now we turn our attention to the case where $\nu=k\a$ for some root $\a$. By the same arguments as in the previous case, the modules of the form $A(L_1,\ldots,L_n)$ where $n\geq 2$ yield all the irreducible modules for $R(k\a)$ except one, unless $\nu=n\d$, when the construction yields all irreducible modules except $f(n)$. It suffices to prove that if $L$ is an irreducible representation of $R(\nu)$ with $L$ not of the form $A(L_1,\ldots,L_n)$ with $n\geq 2$, then $L$ is semicuspidal.

Suppose that $\la$ and $\mu$ are such that $\Res_{\la\mu}L\neq 0$. We need to prove that $\la$ is a sum of roots less than or equal to $\a$ (the result for $\mu$ is similar) and we may suppose that neither of $\la$ and $\mu$ is zero. Let $L_\la\otimes L_\mu$ be an irreducible submodule of $\Res_{\la\mu}L$. By inductive hypothesis $L_\la=A(L_1,\ldots,L_k)$ for some admissible sequence of semicuspidal representations. Suppose that $L_1$ is a $R(m\b)$ module where $\b$ is a root. Then $\Res_{m\b,\nu-m\b}L\neq 0$.
If $\b\preceq\a$, then $\la$ is a sum of roots less than or equal to $\b$ and hence a sum of roots less than or equal to $\a$.

Therefore without loss of generality we may assume that $\la=m\b$ and that $L_\la$ is semicuspidal. For want of a contradiction, assume $\b\succ \a$. We may further assume without loss of generality that $\b$ is the maximal root for which $\Res_{m\b,\nu-m\b}L\neq 0$ for some positive integer $m$. We may further assume that $m$ is as large as possible.

By inductive hypothesis, write $L_\mu=A(M_1,\ldots,M_n)$ where $M_1$ is a $R(k\ga)$-module for some root $\ga$ and positive integer $k$.

Therefore $\Res_{\la+k\ga,\mu-k\ga}L\neq 0$. If $k\ga\neq \mu$, then by maximality of $\b$, $\la+k\ga$ is a sum of roots less than or equal to $\b$. By maximality of $m$, $\ga \prec \b$. By adjunction this implies that $L$ is a quotient of $L_\la\circ M_1\circ\cdots \circ M_n$. As $(L_\la,M_1,\ldots,M_n)$ is an admissible sequence of semicuspidal modules, this is a contradiction.

Therefore $L_\mu$ is semicuspidal, with $\mu=k\ga$. By convexity $\ga\prec\a\prec\b$.
By adjunction there is a nonzero map from $L_\la\circ L_\mu$ to $L$. As $L$ is irreducible, this exhibits $L$ as $A(L_\la,L_\mu)$, a contradiction. This completes the proof.
\end{proof}

\section{Real Cuspidals}

For $i\in I$, there is an automorphism $T_i$ of the entire quantum group $U_q(\g)$ satisfying
\[
T_i \th_j = \sum_{k=0}^{-i\cdot j} (-q)^k \th_i^{(k)} \th_j \th_i^{(-i\cdot j-k)}
\]
for all $i\neq j$.
In the notation of \cite{lusztigbook}, $T_i$ is the automorphism $T'_{i,+}$.

Now we will define the PBW root vectors for the real roots. Let $\a$ be a positive real root and suppose that $\a\prec\delta$. Let $S_\a=\{\b\in \Phi^+\mid
\a-\b\in \N I \}$. Then $S_\a$ is a finite set of roots. By Theorem \ref{3.11}, we can find a word convex order $\prec'$ whose restriction to $S_\a$ agrees with the restriction of $\prec$ to $S_\a$.

By Lemma \ref{finiteinitial} there exists $w\in W$ such that $\Phi(w)=\{\b\in\Phi^+\mid \b\preceq'\a\}$ and a reduced expression $w=s_{i_1}\ldots s_{i_N}$ such that $\a=s_{i_1}s_{i_2}\cdots s_{i_{N-1}} \a_{i_N}$. We define the root vector $E_\a\in\f$ by
\[
 E_\a=T_{i_1}T_{i_2}\cdots T_{i_{N-1}} \th_{i_N}
\]

If $\a$ happens to be greater than $\delta$, then in a similar vein we get a reduced expression but now define $E_\a\in\f$ by
\[
 E_\a=T_{i_1}^{-1}T_{i_2}^{-1}\cdots T_{i_{N-1}}^{-1} \th_{i_N}
\]

In all cases, we then define the dual root vector $E_\a^*=(1-q_\a^2)E_\a\in\f^*$.

A proof that the elements $E_\a$ and $E_\a^*$ are well defined based on \cite[Proposition 40.2.1]{lusztigbook} is possible. Alternatively, this result will follow from Theorem \ref{realdualpbw}.

For $\a\in\Phi^+_{\re}$, let $L(\a)$ be the unique self-dual cuspidal irreducible representation of $R(\a)$. The existence of a cuspidal irreducible module is Theorem \ref{numberofimaginarysemicuspidals} above while the fact that it can be chosen to be self-dual is in \cite[\S 3.2]{khovanovlauda}.

\begin{theorem}\label{realdualpbw}
Let $\a$ be a positive real root. Then $[L(\a)]=E_\a^*$.
\end{theorem}

\begin{proof}
Let $i_1,\ldots,i_N$ be as in the construction of $E_\a$ above. For $1\leq k\leq N$ let $\a_k=s_{i_1}\cdots s_{i_{k-1}} \a_{i_k}$. Then $\a_1\prec\cdots\prec\a_N=\a$.

First we will prove by induction on $n$ for $1\leq n\leq N$ that there exists $x_n\in \f^*_{\Q(q)}$ such that $[L(\a)]=T_{i_1}\cdots T_{i_{n-1}}(x_{n})$.

For the case $n=1$, let $x_1=[L(\a)]$.
Now assume that the result is known for $n=k$ and consider the case $n=k+1$.

By \cite[Ch 38]{lusztigbook} we can write $x_k=\th_{i_k}^* y + T_{i_k}(z)$ where $y,z\in \f^*_{\Q(q)}$. Then
\[
[L(\a)]=(T_{i_1}\cdots T_{i_{k-1}}(\th_{i_k}^*)) (T_{i_1}\cdots T_{i_{k-1}}(y)) + T_{i_1}\cdots T_{i_k}(z).
\]
Since $L(\a)$ is cuspidal and $\a_k\prec \a$, $\Res_{\a_k,\a_N-\a_k}L(\a)=0$. Therefore $[L(\a)]$ is orthogonal to the product $(T_{i_1}\cdots T_{i_{k-1}}(\th_{i_k}^*)) (T_{i_1}\cdots T_{i_{k-1}}(y))$. Since this product is orthogonal to $T_{i_1}\cdots T_{i_k}(z)$, it must be that $(y,y)=0$.

If $\Q(q)$ is embedded into $\R$ by sending $q$ to a sufficiently small real number, then the form $(\cdot,\cdot)$ on $\f_\nu$ is positive definite. Therefore $y=0$. We let $x_{k+1}=z$.


We have now proved the desired preliminary result by the principle of mathematical induction. Applying this when $k=N$, we see that $[L(\a)]=T_{i_1}\cdots T_{i_{n-1}}(x_{N})$ for some $x_N\in (\f^*_{\Q(q)})_{\a_{i_N}}$. Therefore $x_N$ is a scalar multiple of $\th_{i_N}$.

Since $[L(\a)]$ is an element of a $\Z[q,q\inv]$-basis of $\f^*$, the scalar must be a unit in $\Z[q,q\inv]$, thus of the form $\pm q^i$ for some $i\in \Z$. 
Since $[L(\a)]$ is
invariant under the bar involution, $i=0$. The argument of \cite[Lemma 2.35]{kleshchev} shows that the sign is the positive one.
\end{proof}

\begin{proposition}
Let $\a$ be a real root and $n$ be a positive integer. The module $L(\a)^{\circ n}$ is the unique simple semicuspidal representation of $R(n\d)$.
\end{proposition}

\begin{proof}
By Lemma \ref{10.2}, $L(\a)^{\circ n}$ is semicuspidal. Therefore $[L(\a)^{\circ n}]=f(q)[L]$ where $L$ is the unique semicuspidal representation of $R(n\a)$ and $f(q)\in\N[q,q\inv]$. By Theorem \ref{realdualpbw}, $[L(\a)^{\circ n}]=T_{i_1}T_{i_2}\cdots T_{i_{N-1}} (\th_{i_N}^*)^n$ which is indivisible in $\f^*$, hence $L(\a)^{\circ n}$ is irreducible.
\end{proof}

\begin{remark}
This gives the existence of many modules called \emph{real} in the nomenclature of \cite{kkko}.
\end{remark}

\section{Root Partitions}\label{rootpartitions}

Let $S$ be an indexing set for the set of self-dual irreducible semicuspidal representations of $R(n\d)$, for all $n$. It will not be until Theorem \ref{19.7} that we exhibit a bijection between $S$ and $\partition^\W$. We write $L(s)$ for the representation indexed by $s\in S$.

We now introduce the notion of a root partition, which allows us to index irreducibles by a finite collection of real roots (with multiplicities), together with an irreducible semicuspidal imaginary module.
We first define a root partition $\pi$ to be an admissible sequence of self-dual irreducible semicuspidal representations. 

To each root partition $\pi$ we define a function $f_\p\map{\Phi^+_{nd}}{\N}$ where if $f_\p(\a)$ is nonzero then there is a representation of $R(f_\p(\a)\a)$ in $\p$. Given two root partitions $\pi$ and $\sigma$ we say that $\pi<\sigma$ if there exist indivisible roots $\a$ and $\a'$ such that $f_\p(\a)<f_\sigma(\a)$, $f_\p(\a')<f_\sigma(\a')$ and $f_\p(\b)=f_\sigma(\b)$ for all roots $\b$ satisfying either $\b\prec\a$ or $\b\succ\a'$. If $f_\pi=f_\sigma$ we say $\p\sim\sigma$.

Since there is exactly one irreducible semicuspidal representation of $R(n\a)$ for each $n$ and each real root $\a$, we can write the datum of a root partition in a more combinatorial manner. 
Concretely we write a root partition in the form $\pi=(\b_1^{m_1},\ldots,\b_k^{m_k},s,\ga_l^{n_l},\ldots,\ga_1^{n_1})$. Here $k$ and $l$ are natural numbers, $s\in S$, $\b_1,\ldots,\b_k,\ga_1,\ldots,\ga_l$ are the set of real roots on which $f_\p$ is nonzero, $f_\pi(\b_i)=m_i$, $f_\pi(\ga_i)=n_i$ and
\[
\b_1\succ\cdots\succ\b_k \succ\d\succ \ga_l\succ\cdots\succ\ga_1
\] 
When we do have a bijection between $S$ and $\partition^\W$ then we will have a purely combinatorial description of a root partition.

Let $\pi=(\b_1^{m_1},\ldots,\b_k^{m_k},s,\ga_l^{n_l},\ldots,\ga_1^{n_1})$ be a root partition. Let $s_\la=\sum_{i=1}^k{\binom{m_i}{2}}+\sum_{j=1}^l\binom{n_j}{2}$. Define the proper standard module $\overline{\Delta}(\pi)$ to be
\[
\overline{\D}(\pi)=q^{s_\la}L(\b_1)^{\circ m_1}\circ \cdots\circ L(\b_k)^{\circ m_k} \circ L(s)\circ L(\ga_l)^{\circ n_l} \circ \cdots \circ (\ga_1)^{\circ n_1}.
\]

Let $L(\pi)$ be the head of $\overline{\D}(\pi)$. This is an irreducible module by Lemma \ref{irreduciblehead}.

\begin{theorem}\label{main}\cite{kleshchev}
The proper standard modules have the following property.
 \begin{enumerate}
\item Up to isomorphism and grading shift, the set $\{L(\pi)\}$ as $\pi$ runs over all root partitions of $\nu$ is a complete and irredundant set of irreducible $R(\nu)$-modules.
\item The module $L(\pi)$ is self dual, i.e. $L(\pi)^\circledast \cong L(\pi)$.
\item If the multiplicity $[\overline{\D}(\pi):L(\sigma)]$ is nonzero, then $\sigma\leq \pi$. Furthermore $[\overline{\D}(\pi):L(\pi)]=1$.
 \end{enumerate}
\end{theorem}

\begin{proof}
Part (1) is Theorem \ref{fdmain}. For part (2) note that since $L(\p)$ is irreducible, by \cite[\S 3.2]{khovanovlauda} $L(\pi)\dual\cong q^i L(\pi)$ for some $i$. By Lemma \ref{irreduciblehead}(2) and the fact that restriction commutes with duality, $i=0$. Part (3) follows from Lemma \ref{restrictinduct}.
\end{proof}

\section{Levendorskii-Soibelman Formula}\label{lsformula}

By Theorem \ref{main} the classes $[\overline{\D}(\pi)]$ of the proper standard modules is a basis of $\f^*$. We call this the categorical dual PBW basis. Let $\{E_\pi\}$ be the basis of $\f$ dual to this with respect to $\langle\cdot,\cdot\ra$. We shall call this basis the categorical PBW basis. Later we will identify the categorical PBW basis both with a basis coming from a family of standard modules, as well as an algebraically defined basis which generalises the approach of \cite{beck}.


The results in this section are an affine type analogue of the Levendorskii-Soibelman formula \cite[Proposition 5.5.2]{ls}. We refer to both Theorems \ref{ls} and \ref{lsf} as a Levendorskii-Soibelman formula.

\begin{theorem}\label{ls}
Let $\theta,\psi\in\Phi^+_{\re} \cup S$ with $\theta \succ\psi$. Expand $[L(\theta)][L(\psi)]-q^{(\th\cdot\psi)}[L(\psi)][L(\th)]$ in the standard basis
\[
 [L(\theta)][L(\psi)]-q^{(\th\cdot\psi)}[L(\psi)][L(\th)]=\sum_\pi c_\pi [\overline{\D}(\pi)].
\]
If $c_\pi\neq 0$ for some root partition $\pi$ then $\pi< (\theta,\psi)$ where $<$ is the partial order on root partitions from $\S\ref{rootpartitions}$.
\end{theorem}

\begin{proof}
 By Theorem \ref{main}, 
\[
 [L(\th)][L(\psi)]-[L(\th,\psi)]\in\sum_{\pi<(\th,\psi)}\Z[q,q\inv] [L(\pi)].
\]
Applying the bar involution on $\f^*$ yields
\[
 q^{(\th\cdot\psi)}[L(\psi)][L(\th)]-[L(\th,\psi)]\in \sum_{\pi<(\th,\psi)}\Z[q,q\inv] [L(\pi)].
\]
Theorem \ref{main} also shows that
\[
 \sum_{\pi<(\th,\psi)}\Z[q,q\inv] [L(\pi)]=\sum_{\pi<(\th,\psi)}\Z[q,q\inv] [\bar\D(\pi)]
\]
so upon subtraction we obtain the desired result.
\end{proof}

\begin{lemma}\label{orthog}
Let $\sigma$ and $\pi$ be two root partitions. Then $\langle [\overline{\D}(\sigma)],[\overline{\D}(\pi)]\rangle = 0$ unless $\sigma\sim\pi$.
\end{lemma}

\begin{proof}
We have $\langle [\overline{\D}(\sigma)],[\overline{\D}(\pi)]\rangle = \langle [L(\sigma_1)\otimes \cdots \otimes L(\sigma_l)],[\Res_\sigma\overline{\D}(\pi)]\rangle$ which, by Lemma \ref{restrictinduct}, is zero unless $\sigma\leq \pi$.

Also $\langle [\overline{\D}(\sigma)],[\overline{\D}(\pi)]\rangle=\langle [\Res_\pi(\overline{\D}(\sigma)],[L(\pi_1)\circ \cdots \circ L(\pi_k)]\rangle$ which again by Lemma \ref{restrictinduct} is zero unless $\pi\leq \sigma$.
\end{proof}

\begin{lemma}\label{thsimpsi}
If $\th\in\Phi^+_\re\cup S$ then
\[
E_\th \in \operatorname{span}_{\psi\sim \th} [L(\psi)].
\]
\end{lemma}

\begin{proof}
 By definition of $E_\th$, we have $\langle E_\th,[\overline{\D}(\pi)]\rangle = 0$ unless $\pi=\th$. By Lemma \ref{orthog} and the fact that the classes $[\overline{\D}(\sigma)]$ are a basis of $\f^*$, this forces
\[
E_\th \in \operatorname{span}_{\psi\sim \th} [\overline{\D}(\psi)]= \operatorname{span}_{\psi\sim \th} [L(\psi)].
\]\end{proof}

\begin{corollary}\label{improd}
If $\sigma,\pi\in S$ then $E_\sigma E_\pi$ is a linear combination of $E_\tau$ for $\tau\in S$.
\end{corollary}

\begin{proof}
This follows from Lemmas \ref{10.2} and \ref{thsimpsi}.
\end{proof}

\begin{lemma}
Let $\pi=(\pi_1,\ldots,\pi_k)$ be a root partition. Then $E_\pi=E_{\pi_1}\cdots E_{\pi_k}$.
\end{lemma}
\begin{proof}
 By Lemma \ref{thsimpsi}, the element $E_{\pi_1}\cdots E_{\p_k}$ is a linear combination of elements of the form $[\overline{\D}(\sigma)]$ where $\sigma\sim \pi$. Therefore by Lemma \ref{orthog}, $E_{\pi_1}\cdots E_{\p_k}$ is orthogonal to all elements of the from $[\overline{\D}(\eta)]$ where $\eta\not\sim\pi$. For $\eta\sim\pi$, we compute
\[
 \langle E_{\p_1}\cdots E_{\p_k} , [\overline{\D}(\eta)]\rangle = \prod_{i=1}^k \langle E_{\p_i},[L(\eta_i)]\rangle
\]
which is zero unless $\eta=\pi$ in which case it is equal to one. We have shown that the product $E_{\p_1}\cdots E_{\p_k}$ has all the properties which define $E_\pi$, hence is equal to $E_\p$.
\end{proof}

\begin{theorem}\label{lsf}
Let $\theta,\psi\in\Phi^+_{\re} \cup S$ with $\theta \succ\psi$. Then
\[
E_\th E_\psi - q^{(\th\cdot\psi)} E_\psi E_\th \in \sum_{\pi<(\th,\psi)} \Z[q,q\inv]E_\p.
\]
\end{theorem}

\begin{proof}
This is immediate from Theorem \ref{ls} and Lemma \ref{thsimpsi}.
\end{proof}

This yields an algorithm for expanding any monomial in the $E_\th$ in the PBW basis. Namely given a monomial $E_{\ka_1}E_{\ka_2}\cdots E_{\ka_k}$, repeatedly apply the following types of moves:
\begin{itemize}
\item
If $\ka_l \prec \ka_{l+1}$, replace $E_{\ka_l}E_{\ka_{l+1}}$ with $q^{\ka_l\cdot\ka_{l+1}}E_{\ka_{l+1}}E_{\ka_{l}}$ plus the correction term from Theorem \ref{lsf}.
\item If $\ka_l$ and $\ka_{l+1}$ are both imaginary replace the product $E_{\ka_l}E_{\ka_{l+1}}$ with a sum of terms $E_\th$, which is possible by Corollary \ref{improd}.
\end{itemize}

\section{Minimal Pairs}

Let $\a$ be a positive root. Define $S(\a)$ to be the quotient of $R(\a)$ by the two-sided ideal generated by the set of $e_\ii$ such that $e_\ii L=0$ for all semicuspidal modules $L$.

\begin{lemma}
 There is an equivalence of categories between the category of $S(\a)$ modules and the full subcategory of semicuspidal $R(\a)$-modules.
\end{lemma}

\begin{proof}
It is clear from the definition that any semicuspidal $R(\a)$-module is a $S(\a)$-module.

Conversely suppose that $M$ is a $S(\a)$-module. Suppose that $\la=(\la_1,\ldots,\la_l)$ is a root partition such that $q^n L(\la)$ appears as a subquotient of $M$. Then $e_\ii M\neq 0$ for some $\ii$ which is the concatenation of $\ii_1,\ldots,\ii_l$ in $\Seq(\la_1),\ldots,\Seq(\la_l)$ respectively. If $\la\neq \a$, then $e_\ii L=0$ for all semicuspidal $R(\a)$-modules $L$. Therefore $e_\ii$ has zero image in $S(\a)$, contradicting $e_\ii M\neq 0$. Hence all composition factors of $M$ are semicuspidal, so $M$ is semicuspidal.
\end{proof}

\begin{definition}
Let $\a$ be a positive root. A minimal pair for $\a$ is an ordered pair of roots $(\b,\ga)$ satisfying $\a=\b+\ga$, $\ga\prec\b$ and there is no pair of roots $(\b',\ga')$ satisfying $\a=\b'+\ga'$ and $\ga\prec\ga'\prec\b'\prec\b$.
\end{definition}

\begin{lemma}\label{oppositerestrict}
Let $\a$ be a positive root and let $(\b,\ga)$ be a minimal pair for $\a$. Let $L$ be a cuspidal representation of $R(\a)$. Then $\Res_{\ga,\b}L$ is a $S(\ga)\otimes S(\b)$-module.
\end{lemma}

\begin{proof}
Expand $[\Res_{\ga,\b}L]$ in the categorical dual PBW basis
\[
[\Res_{\ga,\b}L]= \sum_{\pi,\sigma} c_{\pi\sigma}E_\pi^* E_\sigma^*.
\]
Then
\[
c_{\pi\sigma}=\langle E_\pi\otimes E_\sigma,[\Res_{\ga,\b}L]\rangle =\langle E_\pi E_\sigma , [L]\rangle.
\]

In the previous section we showed how the Levendorskii-Soibelman formula gave an algorithm for expanding the product $E_\pi E_\sigma$ into the PBW basis. Each term $E_{\ka_1}\cdots E_{\ka_n}$ which appears at some point in this expansion has $\ka_1\succeq\pi_1\succeq \b$ and $\ka_n\preceq \sigma_l\preceq \ga$.

The only PBW basis elements which fail to be orthogonal to $[L]$ are those of the form $E_\a$ if $\a$ is real and $E_s$ with $s\in S$ if $\a$ is imaginary. For such a term to appear, it must arise as a result of applying Theorem \ref{lsf} to a term $E_{\ka_1}E_{\ka_2}$ with $\ka_1+\ka_2=\a$.

We have already showed that $\ka_1\succeq \b$ and $\ka_2\preceq \ga$. To apply the Levendorskii-Soibelman formula we need $\ka_1\prec \ka_2$ and we also know $\ka_1+\ka_2=\a$. Since $(\b,\ga)$ is a minimal pair, this forces $\ka_1=\b$ and $\ka_2=\ga$. Therefore the coefficient $c_{\pi\sigma}$ can only be nonzero if $\pi=\ka_1$ and $\sigma=\ka_2$. Hence $[\Res_{\ga\b}L]$ is a linear combination of elements of the form $[L_\ga]\otimes [L_\b]$ where $L_\ga$ and $L_\b$ are cuspidal representations of $R(\ga)$ and $R(\b)$. This implies that $\Res_{\ga\b}L$ is a $S(\ga)\otimes S(\b)$-module, as required.
\end{proof}

A chamber coweight $\w$ is said to be \emph{adapted} to the convex order $\prec$ if it is a fundamental coweight for the positive system $p(\Phi_{\succ\d})$ in $\Phi_f$. Let $\w$ be such a chamber coweight. Then there exists a root $\a\in p(\Phi_{\succ\d})$ such that $\langle\w,\a\rangle=1$ and $\langle\w,\b\rangle=0$ for all $\b\in p(\Phi_{\succ\d})\setminus\{\a\}$. Let $\w_+=\tilde{\a}$ and $\w_-=\widetilde{-\a}$.
We will always assume that all chamber coweights are adapted to the given convex order.

\begin{lemma}
Let $\a$ be a positive real root which is not simple that does not have a real minimal pair. Then there exists a chamber coweight $\om$ adapted to $\prec$ such that $\a=\om_++n\d$ or $\a=\om_-+n\d$ for some $n\in\N$.
\end{lemma}

\begin{proof}
Since every root which is not simple has a minimal pair, if $\a$ has no real minimal pair it must be that $\a-\d$ is also a root.

Without loss of generality suppose $\a\succ \d$. The $p(\a)$ is a positive root in $\Phi_f$. We have to prove that $p(\a)$ is simple. Suppose for want of a contradiction that $p(\a)=\b+\ga$ for two positive roots $\b$ and $\ga$. Then $\a=\tilde{\b}+\tilde{\ga}+n\d$ for some $n$. If $n\geq 0$ then we can use this expression to write $\a$ as a sum of two roots both greater than $\d$, which proves that $\a$ has a real minimal pair.

Therefore the only case left to consider is if $\tilde{\b}+\tilde{\ga}-2\d$ is a positive root. When writing $\tilde{\b}+\tilde{\ga}$ in the form $n\d+x$ with $x\in\Phi_f$, $n\geq 2$ with equality if and only if $\b$ and $\ga$ are negative under the usual positive system on $\Phi_f$. Therefore it is impossible for $\tilde{\b}+\tilde{\ga}-2\d$ to be a root.
\end{proof}

\section{Independence of Convex Order}

In this section, we prove some results detailing how some modules which a priori depend on the entire convex order $\prec$, only depend on the positive system $p(\Phi_{\prec \d})$.

\begin{theorem}\label{independence}
Let $\w$ be a chamber coweight.
The algebras $S(\omega_+)$ and $S(\om_-)$ only depend on the set $p(\Phi_{\prec \d})$. 
\end{theorem}

\begin{remark}
 For a balanced convex order, this is \cite[Lemma 5.2]{kleshchev}.
\end{remark}

\begin{proof}
It suffices to prove that the simple modules $L(\w_-)$ and $L(\w_+)$ depend only on $p(\Phi_{\prec \d})$.

We write $E_\a^\prec$ for the root vector defined using the convex order $\prec$. Let $\prec$  and $\prec'$ be two convex orders with $p(\Phi_{\prec \d})=p(\Phi_{\prec' \d})$. Without loss of generality we may assume that $\prec$ and $\prec'$ are of word type. 
Label the roots smaller than $\d$ as $\a_1\prec\a_2\prec\cdots$ and $\a_1'\prec'\a_2'\prec'\cdots$. Let $n$ and $N$ be such that 
 \[
\w_-\in\{\a_1,\ldots,\a_n\}\subset \{\a_1',\ldots,\a_N'\}
\]

Let $w$ be the element of $W$ such that $\Phi(w)=\{\a_1,\ldots,\a_n\}$ and let $u\in W$ be such that $\Phi(u)=\{\a_1',\ldots,\a_N'\}$. Then $\Phi(w)\subset \Phi(u)$. 
Hence if we fix a reduced expression for $w$ (in particular the one used to define $E_{\w_-}^\prec$) then there exists a reduced expression for $u$ beginning with this fixed reduced expression for $w$.

By \cite[Prop 40.2.1]{lusztigbook} there exists a subspace $U^+(u)$ of $\f$ which contains $E_{\w_-}^\prec$ and $E_{\w_-}^{\prec'}$. The dimension of $U^+(u)_{\w_-}$ is equal to the number of ways of writing $\w_-$ as a $\N$-linear combination of roots in $\Phi^+(u)$. Any nontrivial expression contradicts the simplicity of $p(\w_-)$, hence this space is one-dimensional, so $E_{\w_-}^\prec$ and $E_{\w_-}^{\prec'}$ are scalar multiples of one another.

By Theorem \ref{realdualpbw}, $(1-q^2)E_{\w_-}$ is the character of the irreducible module $L(\w_-)$, hence this scalar must be one and the module $L(\w_-)$ is the same for the convex orders $\prec$ and $\prec'$. This completes the proof for $\w_-$ and the proof for $\w_+$ is similar.
\end{proof}

\begin{lemma}\label{strongminimalpair}
 Let $(\b,\ga)$ be a minimal pair for $\d$. Let $L_\b$ and $L_\ga$ be cuspidal $R(\b)$ and $R(\ga)$-modules respectively. Then $\Res_{\ga\b}(L_\ga\circ L_\b) \cong L_\ga\otimes L_\b$ and $\Res_{\ga\b}(L_\b\circ L_\ga) \cong q^{-\b\cdot \ga} L_\ga\otimes L_\b$.
\end{lemma}

\begin{proof}
By Lemma \ref{independence}, without loss of generality, assume our convex order $\prec$ is as in Example \ref{tworow}. Thus the only roots between $\ga$ and $\b$ are of the form $\ga+n\d$, $\b+n\d$ or $n\d$.

Consider a nonzero quotient in the Mackey filtration of $\Res_{\ga,\b}(L_\ga\circ L_\b)$. Then we have $\la,\mu,\nu\in\N I$ such that $\la+\mu=\ga$, $\mu+\nu=\b$, $\la$ is a sum of roots less than or equal to $\ga$, $\nu$ is a sum of roots greater than or equal to $\b$, while $\mu$ is both a sum of roots greater than or equal to $\ga$ and a (possibly different) sum of roots less than or equal to $\b$.

Consider $\ga=\la+\mu$ which has been written as a sum of roots less than or equal to $\b$. No roots between $\ga$ and $\b$ can appear in this sum. By convexity of the convex order, the only options are $\mu=0$, $\mu=\ga$ and $\mu=\b$. We will have to show that the last two options are not possible.

So suppose for want of a contradiction that $\mu=\ga$. Then $\nu=\b-\ga=\sum_i \nu_i$ with each $\nu_i$ larger than $\b$. Note that there is at least two terms in this sum as $\b-\ga$ is not a root.

Since $(\ga,\b-\ga)=-4$, there exists an index $j$ such that $(\ga,\nu_j)<0$. Therefore $\ga+\nu_j$ is a root. Now consider
\begin{equation}\label{bgn}
 \b = (\ga+\nu_j) + \sum_{i\neq j} \nu_i.
\end{equation}
By convexity this implies $\ga+\nu_j\prec\b$ and as $\nu_j\succ\b\succ\ga$ it must be that $\ga+\nu_j\succ\ga$. 
The equation (\ref{bgn}) implies $|\ga+\nu_j|<|\b|$. But on the other hand we've classified all roots $\a$ between $\b$ and $\ga$ and none of them satisfy $|\a|<\b$, a contradiction.
The case $\mu=\b$ is handled similarly.

Therefore there is only one term in the Mackey filtration, which is the one where $\mu=0$, whence we obtain the lemma.
\end{proof}


\section{Simple Imaginary Modules}\label{sim}

We start by following \cite{kkk} and defining the $R$-matrices for KLR algebras. First we need to introduce some useful elements of $R(\nu)$.

For $1\leq a< n=|\nu|$ we define elements $\varphi_a\in R(\nu)$ by
\[
  \varphi_a e_\ii =
\begin{cases}
  (\tau_ay_a-y_{a}\tau_a)e_\ii  & \text{if $\ii_a=\ii_{a+1}$,} \\
\tau_ae_\ii & \text{otherwise.}
\end{cases}
\]

These elements satisfy the following properties
\begin{lemma}\cite[Lemma 1.3.1]{kkk}\label{klem}
 \hfill
\begin{enumerate}
\item $\varphi_a^2e_\ii=(
Q_{\nu_a,\nu_{a+1}}(x_a,x_{a+1})+\delta_{\nu_a,\nu_{a+1}}) e_\ii.$
\item
$\{\varphi_k\}_{1\le k<n}$ satisfies the braid relations.
\item
 For $w\in S_n$, let $w=s_{a_1}\cdots s_{a_\ell}$ be a reduced expression of $w$ and
set $\varphi_w=\varphi_{a_1}\cdots\varphi_{a_\ell}$.
Then $\varphi_w$ does not depend on the choice of reduced expressions of $w$.

\item For $w\in S_n$ and $1\le k\le n$, we have $\varphi_w x_k=x_{w(k)}\varphi_w$.
\item
For $w\in S_n$ and $1\le k<n$,
if $w(k+1)=w(k)+1$, then $\varphi_w\tau_k=\tau_{w(k)}\varphi_w$.\label{ga5}
\item
$\varphi_{w^{-1}}\varphi_we_\ii=\prod\limits_{ \substack{a<b,\\ w(a)>w(b)} }
(Q_{\ii_a,\ii_b}(x_a,x_b)+\delta_{\ii_a,\ii_b})e_\ii$. \label{ga6}
\end{enumerate}
\end{lemma}

Let $M$ and $N$ be modules for $R(\la)$ and $R(\mu)$ respectively. Let $(\la,\mu)_n$ be the degree of $\phi_{w[\la,\mu]}$. Define the morphism $R_{M,N}\map{M\circ N}{q^{-(\la,\mu)_n}N\circ M}$ by
\[
R_{M,N}(u\otimes v)=\varphi_{w[\la,\mu]}v\otimes u.
\] 

In \cite{kkk} an algebra homomorphism $\psi_z\map{R(\nu)}{\Q[z]\otimes R(\nu)}$ is constructed where $\psi_z(e_\ii)=e_\ii$, $\psi_z(y_j)=y_j+z$ and $\psi_z(\tau_k)=\tau_k$. If $M$ is an $R(\nu)$-module we define the $R(\nu)$-module $M_z=\psi_z^*(\Q[z]\otimes M)$. The morphism $r_{M,N}\map{M\circ N}{q^{2s-(\la,\mu)_n}N\circ M}$ is now defined by
\[
r_{M,N} = \left( (z-w)^{-s} R_{M_z,N_w} \right) |_{z=w=0}.
\]
where $s$ is the largest possible integer for which this definition is possible. In \cite{kkk} it is shown that $r_{M,N}$ is a nonzero morphism and that these collections of morphisms satisfy the braid relation.

\begin{lemma}\label{risq}
Let $L_1$ and $L_2$ be two irreducible cuspidal representations of $R(\d)$. Then the morphisms $r_{L_1,L_2}$ and $r_{L_2,L_1}$ are inverse to one another.
\end{lemma}

\begin{proof}
By adjunction
\[
\Hom (L_1\circ L_2,L_2\circ L_1) \cong \Hom(L_1\otimes L_2,\Res_{\d,\d} L_2\circ L_1).
\]
As $L_1$ and $L_2$ are cuspidal, the Mackey filtration of $\Res_{\d,\d} (L_2\circ L_1)$ has two nonzero pieces, namely $L_2\otimes L_1$ and $L_1\otimes L_2$. In particular this implies that $\Hom (L_1\circ L_2,L_2\circ L_1)$ is concentrated in degree zero. Since $r_{L_1,L_2}\neq 0$, the integer $s$ in the construction of $r_{L_1,L_2}$ must be equal to $(\delta,\delta)_n/2$.

For $j=1,2$,
pick a nonzero vector $v_j\in L_j$ such that $y_i v_j=0$ for all $i$. 
The morphism $r_{L_2,L_1}r_{L_1,L_2}$ maps $v_1\otimes v_2$ to $\left((z'-z)^{-2s}\varphi_{w[\d,\d]}^2 v_1\otimes v_2\right)|_{z=z'=0}$ where the computation is taking place in $(L_1)_{z}\circ (L_2)_{z'}$ (by abuse of notation, we write $v$ for $1\otimes v\in L_z$). We can compute this using Lemma \ref{klem}(vi). Since $y_i v_j=0$ in $L_j$, we have $y_i v_j=zv_j$ in $(L_j)z$. Then the product on the right hand side of \ref{klem}(vi) acts by the scalar $(z'-z)^{(\delta,\delta)_n}$ on the vector $v_1\otimes v_1\in (L_1)_{z}\circ (L_2)_{z'}$. We've already computed $(\delta,\delta)_n=2s$ and hence $r_{L_2,L_1}r_{L_1,L_2} v_1\otimes v_2=v_1\otimes v_2$.

Since $L_1$ and $L_2$ are irreducible, $v_1\otimes v_2$ generates $L_1\circ L_2$. Therefore $r_{L_2,L_1}r_{L_1,L_2}$ is the identity.
\end{proof}

From the evident maps from $\End(L\circ L)$ to $\End(L^{\circ n})$, 
the morphisms $r_{L,L}$ define $n-1$ elements, denoted $r_1,r_2,\ldots , r_{n-1} \in \End(L^{\circ n})$.
The following result was first noticed in a special case in 
 \cite[Theorem 4.13]{kmr}, and is fundamental to the paper
\cite{imaginaryschurweyl}.

\begin{theorem}\label{14.3}
Let $L$ be an irreducible cuspidal representation of $R(\d)$.
 There is an isomorphism $\End(L^{\circ{n}})\cong \Q[S_n]$ sending $r_i$ to the transposition $(i,i+1)$.
\end{theorem}

\begin{proof}
 By adjunction $\End(L^{\circ n})=\Hom(L^{\otimes n},\Res_{\de,\ldots,\de}L^{\circ n})$. Since $L$ is cuspidal, the Mackey filtration of $\Res_{\de,\ldots,\de}L^{\circ n}$ has exactly $n!$ nonzero subquotients, each isomorphic to $L^{\otimes n}$. Therefore $\dim\End(L^{\circ n})\leq n!$. 

By Lemma \ref{risq}, $r_i^2=1$. The identity $r_ir_j=r_jr_i$ for $|j-i|>1$ is trivial and the braid relation $r_i r_{i+1} r_i= r_{i+1} r_i r_{i+1}$ is a general fact about the morphisms $r_{M,N}$ constructed in \cite{kkk}. This allows us to define $r_w$ for each $w\in S_n$.

Recall that in the proof of Lemma \ref{risq}, we showed that $s=(\d,\d)_n/2$, where $s$ is the integer appearing in the definition of $r_{L,L}$.
Therefore by induction on the length of $w$, using \cite[Proposition 1.4.4(iii)]{kkk}, we obtain
\[
r_w v\otimes\cdots \otimes v - \tau_{\iota(w)}v\otimes \cdots \otimes v \in \sum_{\ell(w')<\ell(\iota(w))}\tau_{w'}L\otimes \cdots \otimes L
\]
where $\iota\map{S_n}{S_{n|\delta|}}$ is the obvious embedding.
Therefore the endomorphisms $r_w$ are linearly independent. 

Since the $r_i$ satisfy the Coxeter relations there is a homomorphism from $\Q[S_n]$ to $\End(L^{\circ n})$. We have just shown it is injective. Surjectivity follows from the dimension estimate in the first paragraph of this proof.
\end{proof}

Let $\w$ be a chamber coweight. Let $L(\w)$ be the head of the module $L(\w_-)\circ L(\w_+)$.

\begin{lemma}\label{reslw}
The module $L(\w)$ is an irreducible module with $L(\w)\dual\cong L(\w)$. Furthermore $\Res_{\w_-,\w_+}L(\w)\cong L(\w_-)\otimes L(\w_+)$.
\end{lemma}

\begin{remark}
The irreducibility of $L(\w)$ is in \cite{tingleywebster} and can also be derived from \cite[Theorem 3.2]{kkko}. Our preference for giving this proof is that we wish to make use of the extra properties of $L(\w)$ that we establish.
\end{remark}

\begin{proof}
Using Theorem \ref{independence} and the convex order from Example \ref{tworow}, we may assume without loss of generality that $(\w_+,\w_-)$ is a minimal pair for $\d$.

For any quotient $Q$ of $L(\w_-)\circ L(\w_+)$ there is, by adjunction, a nonzero morphism from $L(\w_-)\otimes L(\om_+)$ to $\Res_{\w_-,\w_+}Q$ which is injective as the source is irreducible.
Lemma \ref{strongminimalpair} implies that $\Res_{\w_-,\w_+} L(\w_-)\circ L(\w_+)\cong L(\w_-)\otimes L(\w_+)$. By exactness of the restriction functor, this forces the head of $L(\w_-)\circ L(\w_+)$ to be irreducible and furthermore $\Res_{\w_-,\w_+}L(\w)\cong L(\w_-)\otimes L(\w_+)$. The self-duality of $L(\w)$ follows since every simple module is self-dual up to a grading shift, duality commutes with restriction and the modules $L(\w_\pm)$ are self-dual.
\end{proof}

\section{The Growth of a Quotient}\label{growth}

Let $z$ be the element $y_1+\cdots +y_{|\nu|}\in R(\nu)$. It is straightforward to check that $z$ is central.
The following lemma and proof appeared in an early version of \cite{bkm}.

\begin{lemma}\label{centretrick}
 Let $R'(\nu)$ be the subalgebra of $R(\nu)$ generated by $e_\ii$, $\ii\in\Seq(\nu)$, $\tau_i$ and $y_i-y_{i+1}$, $1\leq i<|\nu|$. Then multiplication induces an algebra isomorphism $\Q[z]\otimes R'(\nu)\to R(\nu)$. 
\end{lemma}

\begin{proof}
 An inspection of the presentation (\ref{eq:KLR}) of $R(\nu)$ shows that the set of elements of the form
\[
(y_1-y_2)^{a_1}(y_2-y_3)^{a_2}\cdots (y_{n-1}-y_n)^{a_{n-1}}\tau_w e_\ii
\]
with $a_1,\ldots,a_{n-1}\in \N$, $w\in S_n$ and $\ii\in\Seq(\nu)$ is a spanning set for $R'(\nu)$. Since Theorem \ref{basisofrnu} provides us with a basis of $R(\nu)$, we can see that the collection of elements above forms a linearly independent set, hence is a basis for $R'(\nu)$. We compute
$$ny_n= z + \sum_{i=1}^{n-1}i(y_i- y_{i+1})$$
and thus $y_n$ is in the image of $\Q[z]\otimes R'(\nu)$. Therefore the multiplication map from $\Q[z]\otimes R'(\nu)$ to $R(\nu)$ is surjective. A dimension count using Lemma \ref{basisofrnu} shows that it must be an isomorphism.
\end{proof}

\begin{lemma}\label{scentretrick}
Let $\a$ be a positive root. There is an injection from $\Q[z]$ into the centre of $S(\a)$.
\end{lemma}

\begin{proof}
Let $S'(\a)$ be the quotient of $R'(\a)$ by the two sided ideal generated by all $e_\ii$ such that $e_\ii L=0$ for all cuspidal representations $L$ of $R(\a)$.
Lemma \ref{centretrick} implies that $S(\a)\cong \Q[z]\otimes S'(\a)$. The image of $\Q[z]\otimes \Q$ provides us with our desired central subalgebra.
\end{proof}

Let $\a$ be an indivisible root, $L$ a cuspidal representation of $R(\a)$ and let $(\b,\ga)$ be a minimal pair for $\a$. Let $L''\otimes L'$ be an irreducible subquotient of $\Res_{\ga,\b}L$. By Lemma \ref{oppositerestrict}, $L''$ and $L'$ are cuspidal modules for $R(\ga)$ and $R(\b)$. We will call $(L',L'')$ a minimal pair for $L$.
We inductively define a word $\ii_L\in\Seq(\a)$ as the concatenation $\ii_{L''}\ii_{L'}$.




Let $T(L)$ be the subalgebra of $e_{\ii_L} S(\a) e_{\ii_L}$ generated by $y_1 e_{\ii_L},\ldots,y_{|\a|}e_{\ii_L}$.

\begin{lemma}\label{surjbydefn}
Let $L$ be a cuspidal representation and $(L',L'')$ be a minimal pair for $L$. The inclusion $R(\ga)\otimes R(\b)\to R(\a)$ induces a homomorphism from $T(L'')\otimes T(L')$ to $T(L)$.
\end{lemma}

\begin{proof}
Suppose $x\in \ker(R(\ga)\to S(\ga))$. Consider $x\otimes 1\in R(\ga)\otimes R(\b)\hookrightarrow R(\a)$. On $M\in S(\a)\mods$, $x\otimes 1$ acts in the way it does on
$\Res_{\ga\b}M$, which is a $S(\ga)\otimes S(\b)$-module. Therefore $x\otimes 1$ acts by zero and hence is in the kernel of $R(\a)\to S(\a)$.
\end{proof}

\begin{lemma}\label{zerodim}
Let $L$ be a cuspidal representation of $\a$.
The scheme $\Proj T(L)$ has a unique $\C$-point $[x_1:\cdots :x_{|\a|}]$, namely $x_1=\cdots=x_{|\a|}$.
\end{lemma}

\begin{proof}
We prove this by induction on the height of $\a$. Choose a minimal pair $(\b,\ga)$ for $\a$ and $(L',L'')$ for $L$.
Suppose that $[x_1:\cdots : x_{|\a|}]$ is a $\C$-point of $\Proj T(L)$. Then by Lemma \ref{surjbydefn} $[x_1:\cdots : x_{|\ga|}]$ and $[x_{|\ga|+1}:\cdots:x_{|\a|}]$ are points in $\Proj T(L'')$ and $\Proj T(L')$ respectively.
By inductive assumption, $x_1=\cdots = x_{|\ga|}$ and $x_{|\ga|+1}=\cdots = x_{|\a|}$.

Let $w=w[|\ga|,|\b|]$ and consider the element $\varphi_w^2 e_{\ii_L}$. By Lemma \ref{klem}(vi) it lives in $T(\ii_L)$ and since $\varphi_w^2 e_{\ii_L}=\varphi_w e_{\ii_{L'}\ii_{L''}}\varphi_w$, it lives in the kernel of the map from $R(\a)$ to $S(\a)$. Therefore $\varphi_w^2 e_{\ii_L}$ is zero in $T(\ii_L)$.
Lemma \ref{klem}(vi) writes $\varphi_w^2 e_{\ii_L}$ as a product of elements of the form $x_i-x_j$ where $i\leq |\ga|$ and $j>|\ga|$. Therefore any $\C$-point of $\Proj T(\ii_L)$ has $x_1=\cdots=x_{|\a|}$ as required.
\end{proof}

\begin{theorem}\label{oofone}
 Let $\a$ be an indivisible root. Then $\dim S(\a)_d$ is bounded as a function of $d$.
\end{theorem}

\begin{proof}
Consider a composition series for $S(\a)$ as a $S(\a)$-module. Every composition factor must be cuspidal, so
\begin{equation}\label{sl}
[S(\a)]= \sum_L f_L(q) [L]
\end{equation}
where $f_L(q)\in \N((q))$ and the sum is over irreducible cuspidal representations $L$. For any $\ii\in Seq(\nu)$, we therefore get the equality
\begin{equation}\label{fl}
\dim (e_\ii S(\a))=\sum_L f_L(q) \dim e_\ii L.
\end{equation}

Pick an irreducible cuspidal representation $L$ and let $\ii_L$ be the corresponding word in $\Seq(\nu)$. By Lemma \ref{zerodim} and the theory of the Hilbert polynomial, $\dim T(\ii_L)_d$ is a bounded function of $d$. From Theorem \ref{basisofrnu} we see that $e_{\ii_L} S(\a)$ is finite over $T(\ii_L)$ and hence $\dim (e_{\ii_L} S(\a))_d$ is a bounded function of $d$. 

We take $\ii=\ii_L$ in (\ref{fl}) and since $e_{\ii_L}L\neq 0$, the Laurent series $f_L(q)=\sum_d f_L^{(d)} q^d$ has $f_L^{(d)}$ a bounded function of $d$. Equation (\ref{sl}) completes the proof.
\end{proof}

\section{An Important Short Exact Sequence}

Let $\a$ be a real root. Define $\D(\a)$ to be the projective cover of $L(\a)$ in the category of $S(\a)$-modules. Let $\om$ be a chamber coweight. Define $\D(\om)$ to be the projective cover of $L(\om)$ in the category of $S(\d)$-modules.

\begin{lemma}\label{mpses}
Let $\a$ be an indivisible root.
Suppose that $(\b,\ga)$ is a minimal pair for $\a$. Let $\D_\b$ and $\D_\ga$ be finitely generated projective $S(\b)$ and $S(\ga)$-modules.
Then there is a short exact sequence
\[
0 \to q^{-\b\cdot \ga}\D_\b\circ \D_\ga\to \D_\ga\circ \D_\b\to C\to 0
\] for some projective $S(\a)$-module $C$.
\end{lemma}

\begin{proof}
By adjunction,
\[
\Hom(q^{-\b\cdot \ga}\D_\b\circ \D_\ga, \D_\ga\circ \D_\b)\cong \Hom(q^{-\b\cdot\ga}\D_\b\otimes \D_\ga,\Res_{\b\ga}\Delta_\ga\circ\Delta_\b).
\]
Since the modules $\D_\ga$ and $\D_\b$ are cuspidal, the Mackey filtration of $\Res_{\b\ga}\Delta_\ga\circ\Delta_\b$ has only one nonzero term, yielding an isomorphism
$$\Res_{\b\ga}\Delta_\ga\circ\Delta_\b\cong q^{-\b\cdot \ga}\Delta_\b\otimes \Delta_\ga.$$
Let $\phi\map{q^{-\b\cdot\ga}\D_\b\circ \D_\ga}{\D_\ga\circ \D_\b}$ be the image of the identity map on $q^{-\b\cdot\ga}\D_\b\otimes\D_\ga$ under the isomorphisms discussed above.

This map $\phi$ satisfies
\begin{equation}\label{explicitphi}
\phi(1\otimes (v_\b\otimes v_\ga)) = \tau_{w[\b,\ga]} 1\otimes (v_\ga\otimes v_\b)
\end{equation}
for all $v_\b\in \D_\b$ and $v_\ga\in \D_\ga$.

There are filtrations of $\D_\b$ and $\D_\ga$ where each successive subquotient is an irreducible cuspidal module for $R(\b)$ or $R(\ga)$ respectively. This induces a pair of filtrations on $\D_\b\otimes \D_\ga$ and $\D_\ga\otimes \D_\b$ where the successive subquotients are of the form $L_\b \circ L_\ga$ or $L_\ga\circ L_\b$ for cuspidal irreducible representations $L_\b$ and $L_\ga$ of $R(\b)$ and $R(\ga)$.

From the explicit formula (\ref{explicitphi}), we see that $\phi$ induces a morphism $\bar\phi$ on each subquotient $\bar\phi\map{q^{-\b\cdot \ga}L_\b\circ L_\ga}{L_\ga \circ L_\b}$ satisfying
\[
\bar\phi(1\otimes (v_\b\otimes v_\ga)) = \tau_{w[\b,\ga]} 1\otimes (v_\ga\otimes v_\b).
\]

By Theorem \ref{main}, the module $L_\b\circ L_\ga$ has an irreducible head $A(L_\b,L_\ga)$. Since $(\b,\ga)$ is a minimal pair, all other composition factors are cuspidal. Taking duals, $q^{\b\cdot\ga}L_\ga\circ L_\b$ has $A(L_\b,L_\ga)$ as its socle with all other composition factors cuspidal.

The morphism $\bar\phi$ therefore sends the head of $q^{-\b\cdot \ga}L_\b\circ L_\ga$ onto the socle of $L_\ga\circ L_\b$. Hence $\phi$ induces a bijection between all occurrences of non-cuspidal subquotients as sections of filtrations of $q^{-\b\cdot \ga}\D_\b\circ \D_\ga$ and $\D_\ga\circ \D_\b$. This shows that $\ker\phi$ and $\coker\phi$ are both cuspidal $R(\a)$-modules.

Suppose for want of a contradiction that $\ker\phi$ is nonzero. It is a submodule of 
 the finitely generated module $q^{-\b\cdot\ga}\D_\b\circ \D_\ga$. By \cite[Corollary 2.11]{khovanovlauda}, $R(\a)$ is Noetherian and hence $\ker\phi$ is finitely generated.

As $\ker\phi$ is cuspidal it is a $S(\a)$-module, so by Theorem \ref{oofone} we deduce that $\dim(\ker\phi)_d$ is bounded as a function of $d$.

The adjunction (\ref{oppositeadjunction}) yields a canonical nonzero map from $\Res_{\ga,\b}\ker\phi$ to $\D_\ga\otimes \D_\b$. If $X$ is the image of this map then we have $\dim X_d$ is a bounded function of $d$.

The modules $\D_\b$ and $\D_\ga$ are free over the central subalgebra $\Q[z]$ of $S(\b)$ and $S(\ga)$. Therefore $\D_\b\otimes \D_\ga$ is a free $\Q[z_1,z_2]$-module. Hence there are no nonzero submodules $M$ of $\D_\b\otimes \D_\ga$ for which $\dim M_d$ is a bounded function of $d$. This is a contradiction, implying $\phi$ is injective.

Now let $L$ be a cuspidal $R(\a)$-module. We apply $\Hom(-,L)$ to the short exact sequence
\[
 0\to q\D_\b\circ \D_\ga \xrightarrow{\phi} \D_\ga\circ \D_\b \to \coker\phi\to 0.
\]
and obtain a long exact sequence. As $\Res_{\b,\ga}L=0$, we have
\[
 \Ext^i(\D_\b\circ \D_\ga,L) = \Ext^i(\D_\b\otimes \D_\ga,\Res_{\b,\ga}L)=0.
\]
Therefore our long exact sequence degenerates into a sequence of isomorphisms
\begin{equation}\label{uno}
 \Ext^i(\coker\phi,L)\cong \Ext^i(\D_\ga\circ \D_\b,L)
\end{equation}
and by adjunction we have
\begin{equation}\label{duo}
 \Ext^i(\D_\ga\circ \D_\b,L)\cong \Ext^i(\D_\ga\otimes \D_\b,\Res_{\ga,\b}L).
\end{equation}

Lemma \ref{oppositerestrict} shows that $\Res_{\ga,\b}L$ is a $S(\ga)\otimes S(\b)$-module.
Since $\D_\ga\otimes \D_\b$ is a projective $S(\ga)\otimes S(\b)$-module, we derive that $\Ext^1(\D_\ga\otimes \D_\b,\Res_{\ga,\b}L)=0$. Tracing through the above isomorphisms yields $\Ext^1(\coker\phi,L)=0$ and therefore $\coker\phi$ is a projective $S(\a)$-module.
\end{proof}

\section{Cuspidal Representations of $R(\d)$}

We first explain the intertwined logical structure of this section and the following one. Each statement in Section 18 involves a positive root $\a$. We prove all the results in this section under an assumption that the results in Section 18 are known for all roots $\a$ of height less than the height of $\delta$. The reader will not be worried about the forward references once the logical structure of Section 18 is known.

The results of Section 18 will be proved by a simultaneous induction on the height of the root $\a$. In particular, when Theorem 18.1 is proved for a root $\a$, it will be safe to assume that Theorem 18.2 is known for all roots of smaller height. There are references to the results of this section in Section 18. However they only appear when the root $\a$ under question is of height at least that of $\delta$. Thus there is no circularity and the argument is valid.

Let $\omega$ be a chamber coweight. Recall from \S \ref{sim} that $L(\omega)$ is the head of the module $L(\om_-)\circ L(\om_+)$ and is irreducible. Let $\D(\omega)$ be the projective cover of $L(\omega)$ in the category of $S(\d)$-modules. We caution the reader that while $L(\w)$ will depend only on the chamber coweight $\w$ (as in \cite{tingleywebster}), the module $\D(\w)$ will depend not just on $\w$ but also on the positive system $p(\Phi_{\succ\d})$.

\begin{theorem}\label{omegases} Let $\w$ be a chamber coweight.
There is a short exact sequence
\[
0 \to q^2 \D(\om_+)\circ \D(\om_-)\to \D(\om_-)\circ \D(\om_+)\to \D(\om)\to 0.
\]
\end{theorem}

\begin{proof}
As in the proof of Lemma \ref{reslw}, we may assume without loss of generality that $(\w_+,\w_-)$ is a minimal pair for $\d$.

By Lemma \ref{mpses}, there is
a short exact sequence
\begin{equation}\label{cses}
0 \to q^2 \D(\om_+)\circ \D(\om_-)\to \D(\om_-)\circ \D(\om_+)\to C\to 0
\end{equation}
for some projective $S(\d)$-module $C$.

As $C$ is cuspidal, $\Res_{\w_+\w_-}C=0$. By adjunction, this implies that  $\Ext^i(q^2\D(\w_+)\circ \D(\w_-),C)=0$. From the long exact sequence obtained by applying $\Hom(-,C)$ to (\ref{cses}), we therefore get an isomorphism
\begin{equation}\label{1}
 \End(C) \cong \Hom ( \D(\w_-)\circ\D(\w_+),C).
\end{equation}

By Lemma \ref{strongminimalpair} and adjunction, 
\[
 \Ext^1(\D(\w_-)\circ \D(\w_+),q^2 \D(\w_+)\circ\D(\w_-) )=
\Ext^1(\D(\w_-)\otimes \D(\w_+),q^4 \D(\w_-)\otimes \D(\w_+))
\] 
which is zero since $\D(\w_+)\otimes \D(\w_-)$ is a projective $S(\w_+)\otimes S(\w_-)$-module. From the long exact sequence obtained by applying $\Hom(\D(\w_-)\circ\D(\w_+),-)$ to (\ref{cses}), we therefore have a surjection
from $\End(\D(\w_-)\circ \D(\w_+))$ onto $\Hom(\D(\w_-)\circ\D(\w_+),C)$.

Again we apply Lemma \ref{strongminimalpair} and adjunction to obtain
\[
\End(\D(\w_-)\circ\D(\w_+))\cong \End(\D(\w_-)\otimes \D(\w_+))
\]
By Theorem \ref{premorita} this is isomorphic to $\Q[x,y]$ with $x$ and $y$ in degree 2.
Concentrating our attention to degree zero, we obtain $\End(C)_0\cong \Q$. Therefore $C$ is  indecomposable.

The module $L(\omega)$ is by construction a quotient of $\D(\w_-)\circ \D(\w_+)$. Since it is cuspidal, the same argument that produced the isomorphism (\ref{1}) yields an isomorphism
\[
\Hom(C,L(\w))\cong \Hom(\D(\w_-)\circ \D(\om_+),L(\w)).
\]
Therefore
$L(\om)$ is a quotient of $C$. Since $C$ is an indecomposable projective $S(\d)$-module, it must be that $C$ is the projective projective cover of $L(\om)$.
\end{proof}

\begin{corollary}\label{stdim}
Let $\om$ be a chamber coweight. Then $[\D(\om)]\in\f$ and when specialised to $q=1$ is equal to $h_\w \otimes t$.
\end{corollary}

\begin{proof}
This is immediate from Theorems \ref{realpbw} and \ref{omegases}.
\end{proof}

As a consequence we also obtain the following theorem, which also appears in \cite{tingleywebster}.

\begin{theorem}\label{imaginaryclassification}
The set of all modules $L(\omega)$, as $\omega$ runs over the chamber coweights adapted to the convex order $\prec$, is a complete list of the cuspidal irreducible representations of $R(\d)$.
\end{theorem}

\begin{proof}
Corollary \ref{stdim} shows that the modules $\D(\om)$ are a complete set of indecomposable projective modules for $S(\d)$. In the last paragraph of the proof of Theorem \ref{omegases}, we showed that the module $L(\w)$ is a quotient of $\D(\om)$. We also proved that $L(\om)$ is simple in Lemma \ref{reslw}. Therefore the set of such $L(\w)$ is a complete set of irreducible cuspidal representations of $R(\d)$.
\end{proof}

Let $\{n_\w\}_{\w\in\Om}$ be a sequence of natural numbers. Lemma \ref{risq} shows that the induced product
\[
\dct{w\in \W} L(\om)^{\circ n_\w}
\] is independent of the order of the factors.

Now we know the modules $L(\w)$ are pairwise nonisomorphic, we can use
 the same argument as in Theorem \ref{14.3} to obtain a natural isomorphism
\begin{equation}\label{endsn}
\End \left( \dct{w\in \Om} L(\om)^{n_\w} \right)  \cong \bigotimes_{w\in \Om} \Q[S_{n_\w}].
\end{equation}

If $\{m_\w\}_{\w\in\Om}$ and $\{n_\w\}_{\w\in\Om}$ are two sequences of natural numbers then there is a natural inclusion
\begin{equation}\label{endinclusion}
\End \left( \dct{w\in \Om} L(\om)^{m_\w} \right) \otimes \End \left( \dct{w\in \Om} L(\om)^{n_\w} \right) \hookrightarrow \End \left( \dct{w\in \Om} L(\om)^{m_\w+n_\w} \right)
\end{equation}
which, under the isomorphism (\ref{endsn}) is the tensor product of the natural inclusions
\begin{equation}\label{symmetricinclusion}
\Q[S_{m_\w}]\otimes \Q[S_{n_\w}] \hookrightarrow \Q[S_{m_\w +n_\w}].
\end{equation}

If $\om$ is a chamber coweight and $\la$ is a partition of $n$, we define 
\[
 L_\w(\la) = \Hom_{\Q[S_n]} ( S^\la, L(\w)^{\circ n} )
\]
where $S^\la$ is the Specht module for $S_n$.

Let $\uline{\la}=\{\la_\w\}_{\w\in\W}$ be a multipartition. Then we define 
\[
L(\uline{\la})=\dct{\w\in \W} L_\w(\la_\w)=\Hom_{\otimes\Q[S_{n_\w}]} \left(\bigotimes_{\w\in\W} S^{\la_\w},\dct{\w\in \Om}L(\w)^{\circ n_\w}\right).
\]

Define the multi-Littlewood-Richardson coefficients by $$c^{\unu}_{\ula\umu}=\prod_{\w\in\W} c_{\la_\w\mu_\w}^{\nu_\w}$$ where $c_{\la\mu}^\nu$ is the ordinary Littlewood-Richardson coefficient, which we take to be zero if $|\nu|\neq |\la|+|\mu|$.

\begin{theorem}\label{lrl}
The family of modules $L(\ula)$ enjoy the following properties under induction and restriction:
\[
 L(\ula)\circ L(\umu) = \bigoplus_{\unu} L(\unu)^{ \oplus c^{\unu}_{\ula\umu} }
\]
\[
 \Res_{k\d,(n-k)\d} L(\unu) = \bigoplus_{\ula\vdash k,\umu\vdash n-k}  L(\ula)\otimes L(\umu)^{\oplus c^{\unu}_{\ula\umu} }
\]
\end{theorem}

\begin{proof}
 This follows from the observation above that the inclusions (\ref{endinclusion}) and (\ref{symmetricinclusion}) are equivalent under the isomorphism (\ref{endsn}), together with the known formulae for the induction and restriction of Specht modules for the inclusions $S_m\times S_n\to S_{m+n}$.
\end{proof}

As a particular case of Theorem \ref{lrl}, we have
\begin{equation}\label{lrr}
\Res_{k\d,(n-k)\d} L_\w (1^n) \cong L_\w(1^k) \otimes L_\w(1^{n-k}).
\end{equation}

\section{Homological Modules}

See the beginning of the previous section for a discussion of the inductive structure of the arguments in this section.

\begin{theorem}\label{14.1}
Let $\a$ be an indivisible positive root. Let $\D$ and $L$ be $S(\a)$-modules with $\D$ projective. Then for all $i>0$,
\[
\Ext^i(\D,L) =0.\]
\end{theorem}
We remind readers that these Ext groups are taken in the category of $R(\a)$-modules which makes this result nontrivial.

\begin{proof}
Let $(\b,\ga)$ be a minimal pair for $\a$. If $\a$ is a real root, then by the inductive hypothesis applied to Theorem \ref{realpbw} and Corollary \ref{stdim}, there exist projective $S(\b)$ and $S(\ga)$-modules, $\D_\b$ and $\D_\ga$ such that $[\D_\b][\D_\ga]\neq q^{\b\cdot \ga}[\D_\ga][\D_\b]$. Therefore in the short exact sequence of Lemma \ref{mpses}, $C$ is a nonzero direct sum of copies of $\D(\a)$.

If $\a$ is imaginary, then without loss of generality assume that $\D$ is indecomposable projective, hence isomorphic to $\D(\om)$ for some $\om$. Then we use the short exact sequence of Theorem \ref{omegases} and so in all cases we have a short exact sequence
\begin{equation}\label{a}
0 \to q^{-\b\cdot \ga}\D_\b\circ \D_\ga\to \D_\ga\circ \D_\b\to C\to 0
\end{equation}
and it suffices to prove that $\Ext^i(C,L)=0$ for all cuspidal $R(\a)$-modules $L$.

By adjunction there is an isomorphism
\[
\Ext^i(\D_\ga\circ \D_\b,L)\cong \Ext^i(\D_\ga\otimes \D_\b,\Res_{\ga\b}L).
\]
Lemma \ref{oppositerestrict} shows that $\Res_{\ga\b}L$ is a $S(\ga)\otimes S(\b)$-module. Thus by inductive hypothesis we know that this Ext group is zero.

On the other hand, the group $\Ext^{i-1}(q^{-\b\cdot \ga}\D_\b\circ \D_\ga,L)$ is zero by adjunction and the cuspidality of $L$.

Now consider the short exact sequence (\ref{a}) and apply $\Hom(-,L)$ to get a long exact sequence of Ext groups.
In the long exact sequence the group $\Ext^i(C,L)$ is sandwiched between two groups which we have shown to be zero, hence must be zero itself.
\end{proof}

\begin{theorem}\label{realpbw}
Let $\a$ be a real root.
Inside $\f^*_{\Z((q))}$ we have $[\D(\a)]=E_\a$.
\end{theorem}

\begin{proof}
%
%
By Theorem \ref{14.1}, $\langle[\D(\a)],[L(\a)]\rangle=1$. We know that $\D(\a)$ only has $L(\a)$ appearing as a composition factor, and by Theorem \ref{realdualpbw}, $[L(\a)]=E_\a^*$. Therefore $\D(\a)$ is a scalar multiple of $E_\a$. 
By \cite[Proposition 38.2.1]{lusztigbook}, the automorphisms $T_i$ preserve $(\cdot,\cdot)$, hence $\langle E_\a, E_\a^* \rangle =1$ and the scalar is 1.
\end{proof}

\begin{theorem}\label{premorita}
Let $\a$ be a real root.
The endomorphism algebra of $\D(\a)$
 is isomorphic to $\Q[z]$, where $z$ is in degree two. 
\end{theorem}

\begin{proof}
As $\D(\a)$ is the projective cover of $L(\a)$ which is the unique simple $S(\a)$-module, the dimension of $\End(\D(\a))$ is equal to the multiplicity of $L(\a)$ in $\D(\a)$. Theorems \ref{realdualpbw} and \ref{realpbw} tell us that $[\D(\a)]=E_\a$ and $[L(\a)]=E_\a^*$. Since $E_\a^*=(1-q^2)E_\a$, we have $\dim\End(\D(\a))=(1-q^2)\inv$.

There is an injection from the centre of $S(\a)$ into $\End(\D(\a))$. By Lemma \ref{scentretrick}, there is an injection from $\Q[z]$ into $\End(\D(\a))$. A dimension count shows that this injection must be a bijection, as required.
\end{proof}

\begin{corollary}\label{morita}
Let $\a$ be a positive real root. Then the algebras $S(\a)$ and $\Q[z]$ are graded Morita equivalent.
\end{corollary}
\begin{proof}
The module $\D(\a)$ is a projective generator for the category of $S(\a)$-modules and its endomorphism algebra is $\Q[z]$.
\end{proof}

\section{Standard Imaginary Modules}

\begin{lemma}\label{negativehoms}
Let $d\leq 0$ be an integer and let $\om$ and $\om'$ be two chamber coweights. Then \[
\dim\Hom(\D(\w),\D(\w'))_d = \begin{cases} \Q & \text{if $d=0$ and $\w=\w'$,} \\
0 & \text{otherwise}.
\end{cases}
\]
\end{lemma}

\begin{proof}
Since $\D(\om)$ is a projective $S(\d)$-module, the dimension of $\Hom(\D(\om),\D(\om'))$ is equal to the multiplicity of $L(\w)$ in $\D(\om')$.

We have
\[
\langle [L(\om)],[L(\om')] \rangle \in \delta_{\om\w'} + q \Z[[q]]
\]
and by Lemma \ref{14.1}, the bases $\{\D(\om)\}$ and $\{L(\om)\}$ are dual bases for the subspace of $\f_\d$ spanned by the cuspidal modules. Therefore
\[
 [\D(\w)]\in [L(\w)] + \sum_{x\in \Omega} q\Z[[q]]\cdot [L(x)]
\]
which shows the desired properties of the multiplicities.
\end{proof}

\begin{lemma}\label{impr}
 The module $\dct{\om\in\Om}\D(\w)^{\circ n_\w}$ is a projective object in the category of $S(n\d)$-modules.
\end{lemma}

\begin{remark}
We choose an arbitrary ordering of the factors in $\circ_{\om\in\Om}\D(\w)^{\circ n_\w}$. Lemma \ref{deltascommute} below shows that this choice of ordering is immaterial.
\end{remark}

\begin{proof}
 Let $L$ be a semicuspidal $R(n\d)$-module. 
Therefore $\Res_{\d,\ldots,\d}L$ is a $S(\d)\otimes\cdots\otimes S(\d)$-module. By adjunction
\[
\Ext^1 (\dct{\om\in\Om}\D(\w)^{\circ n_\w},L)=\Ext^1(\bigotimes_{w\in \Om} \D(\om)^{\otimes n_\w},\Res_{\d,\ldots,\d}L)
\]
and since each $\D(\om)$ is a projective $S(\d)$-module, this $\Ext^1$ group is trivial, as required.
\end{proof}

\begin{lemma}\label{deltascommute}
Let $\om$ and $\om'$ be two chamber coweights. Then $\D(\om)\circ \D(\om')\cong \D(\om')\circ \D(\om)$.
\end{lemma}

\begin{proof}
We assume that $\w\neq \w'$ as otherwise the result is trivial. By Lemma \ref{negativehoms} and a computation using adjunction and the Mackey filtration, we compute $\End(\D(\w)\circ\D(\w'))_0\cong \Q$. Hence $\D(\w)\circ\D(\w')$ is indecomposable.
By Lemma \ref{impr} the module $\D(\om)\circ \D(\om')$ is a projective $S(2\d)$-module which surjects onto $L(\om)\circ L(\om')$, hence is the projective cover of $L(\w)\circ L(\w')$ in the category of $S(2\d)$-modules. By Lemma \ref{risq}, $L(\w)\circ L(\w')\cong L(\w')\circ L(\w)$, hence their projective covers are isomorphic.
\end{proof}

\begin{theorem}
Let $\{m_\w\}_{\w\in\W}$ and $\{n_\w\}_{\w\in\W}$ be two collections of natural numbers with $\sum_\w m_\w=\sum_\w n_\w$ and let $d\leq 0$ be an integer. Then
\[
\Hom ( \dct{\w\in\W} \D(\om)^{\circ m_\w},\dct{\w\in\W}\D(\w)^{\circ n_\w})_d\cong 
\begin{cases}
\bigotimes_{\w\in\W} \Q[S_{n_\w}] &\text{if $m_\w=n_\w$ for all $\w$ and $d=0$} \\
0 & \text{otherwise}
\end{cases}
\]
\end{theorem}

\begin{proof}
The Mackey filtration for $\Res_{\d,\ldots,\d}(\ic{\w\in\W}\D(\w)^{\circ n_\w})$ has $(\sum_{\w}n_\w)!$ nonzero subquotients, each a tensor product of projective $S(\d)$-modules where the factor $\D(\om)$ appears $n_\w$ times.

Therefore the filtration splits, and by Lemma \ref{negativehoms} and adjunction, the Hom space under question is zero unless $m_\w=n_\w$ for all $\w$ and $d=0$. Furthermore in this case its dimension is $\prod_w n_\w!$.

Since $\circ_{\w\in\W}\D(\w)^{\circ n_\w}$ is a projective $S(n\d)$-module and $\circ_{\w\in\W}L(\w)^{\circ n_\w}$ is a quotient of $\circ_{\w\in\W}\D(\w)^{\circ n_\w}$, every endomorphism of $\circ_{\w\in\W}L(\w)^{\circ n_\w}$ lifts to an endomorphism of $\circ_{\w\in\W}\D(\w)^{\circ n_\w}$. From the dimension counts in the previous paragraph and (\ref{endsn}), this lift is unique in degree zero and hence we get an algebra isomorphism
\[
\End(\dct{\w\in\W}\D(\w)^{\circ n_\w})_0 \cong \End(\dct{\w\in\W}L(\w)^{\circ n_\w}).
\]
So the result follows from (\ref{endsn}).
\end{proof}

For a multipartition $\uline{\la}=\{\la_\w\}_{\w\in\W}$ where each $\la_w$ is a partition of $n_\w$, we define
\[
\D(\uline{\la}) = \Hom_{\otimes \Q[S_{n_\w}]} (\otimes S^{\la_w},\dct{\w\in\W}\D(\w)^{\circ n_\w})
\]

\begin{theorem}
The modules $\D(\la)$ behave in the following way under induction and restriction.
\[
\D(\ula)\circ \D(\umu)\cong \bigoplus_{\unu}\D(\unu)^{\oplus c^{\unu}_{\ula\umu}}
\]
\[
 \Res_{k\d,(n-k)\d} \D(\unu) = \bigoplus_{\ula\vdash k,\umu\vdash n-k}  \D(\ula)\otimes \D(\umu)^{\oplus c^{\unu}_{\ula\umu} }
\]
\end{theorem}

\begin{proof}
 The proof is the same as that of Theorem \ref{lrl}
\end{proof}

Let $f_\la$ be the dimension of the Specht module $S^\la$ and for a multipartition $\ula=\{\la_\w\}_{\w\in\W}$, let $f_\ula=\prod_{\w} f_{\la_w}$.

As a $\Q[S_n]$-module, $\Q[S_n]$ decomposes as $\Q[S_n]=\oplus_{\la} (S^\la) ^{\oplus f_\la}$. Therefore we obtain the decomposition
\begin{equation}\label{decom}
 \dct{w\in \W}\D(\w)^{\circ n_w} \cong \bigoplus_{\uline{\la}\vdash n} \D(\uline{\la}) ^{\oplus f_{\uline{\la}}}.
\end{equation}

\begin{lemma}\label{19.4}
Let $\uline{\la}$ be a multipartition of $n$. The module $\D(\uline{\la})$ is indecomposable.
\end{lemma}

\begin{proof}
 From the decomposition (\ref{decom}) we obtain inclusions
\begin{equation}\label{inclusion}
 \bigoplus_\ula \Mat_{f_\ula}(\Q) \subset \bigoplus_\ula\Mat_{f_\ula}(\End(\D(\uline{\la}))\subset\End(\dct{\w\in\W}\D(\w)^{\circ n_\w}).
\end{equation}
Comparing dimensions shows that these inclusions are isomorphisms in degree zero. Therefore $\End(\D(\uline{\la}))_0$ is isomorphic to $\Q$, hence $\D(\ula)$ is indecomposable.
\end{proof}

\begin{lemma}\label{19.5}
 If $\ula\neq \uline{\mu}$, then $\D(\ula)$ is not isomorphic to any grading shift of $\D(\uline{\mu})$.
\end{lemma}

\begin{proof}
 Let $i\leq 0$ be an integer. The inclusions in (\ref{inclusion}) are all isomorphisms in degrees less than or equal to zero. Therefore $\Hom(\D(\la),\D(\mu))_i=0$ and thus $\D(\la)$ is not isomorphic to $q^i \D(\mu)$. Similarly $\D(\mu)$ is not isomorphic to $q^i \D(\la)$.
\end{proof}

\begin{theorem}
The set $\D(\underline{\la})$ is a complete set of indecomposable projective $S(n\d)$-modules.
\end{theorem}

\begin{proof}
The module $\D(\uline{\la})$ is a direct summand of $\circ_{\w\in\W} \D(\w)^{\circ n_\w}$ which is projective by Lemma \ref{impr}, hence $\D(\ula)$ is projective. Lemmas \ref{19.4} and \ref{19.5} ensure that the set $\{\D(\ula)\}$ is an irredundant set of indecomposable projective $S(n\d)$-modules, up to a grading shift. The number of indecomposable projective $S(n\d)$-modules is equal to the number of irreducible semicuspidal $R(n\d)$-modules. This number is known by Theorem \ref{numberofimaginarysemicuspidals}, hence we have found all of the indecomposable projectives.
\end{proof}

\begin{theorem}\label{19.7}
The set $\{L(\underline{\la})\}_{\ula\vdash n}$ is a complete set of self-dual irreducible $S(n\d)$-modules.
\end{theorem}

\begin{proof}
The set $\D(\uline{\la})$ is a complete set of indecomposable projectives, so the set $\operatorname{hd}\D(\uline\la)$ is a complete set of irreducible $S(n\d)$-modules. Since $\D(\uline\la)$ surjects onto $L(\uline\la)$, the set $\hd(L(\uline\la))$ is a complete set of irreducible $S(n\d)$-modules.
So it suffices to prove that $L(\uline\la)$ is irreducible. 

Let $X$ be a simple submodule of $L(\uline\la)$. Then $X$ is semicuspidal so is of the form $\hd(L(\uline\mu))$ for some multipartition $\uline\mu$. Therefore we get a nonzero morphism from $L(\uline\mu)$ to $L(\uline{\la})$. From the decomposition $\dct{\w\in\W}L(\w)^{\circ n_\w}\cong\oplus_\ula  \D(\ula)^{\oplus f_{\ula}}$ we obtain inclusions
\begin{equation}\label{inclusion2}
 \bigoplus_{\ula\vdash n} \Mat_{f_\ula}(\Q) \subset \bigoplus_{\ula\vdash n}\Mat_{f_\ula}(\End(L(\ula))\subset\End(\dct{\w\in\W}L(\w)^{\circ n_\w}).
\end{equation}
Comparing dimensions shows that these inclusions are equalities and hence all morphisms from $L(\uline\mu)$ to $L(\uline\la)$ are either zero or isomorphisms. Hence $L(\uline{\la})$ must be irreducible, as required. The self-duality of $L(\ula)$ is immediate from the self-duality of $L(\om)$ and (\ref{dualofaproduct}).
\end{proof}

\begin{theorem}\label{14.1nd}
Let $\ula$ and $\umu$ be two multipartitions. Then
\[
\Ext^i(\D(\uline{\la}),L(\uline{\mu})) = \begin{cases}
\Q &\text{if $\uline{\la}=\uline{\mu}$ and $i=0$}, \\
0 &\text{otherwise}.
\end{cases}
\]
\end{theorem}

\begin{proof}
In the course of proving Theorem \ref{19.7}, the module $\D(\uline{\la})$ was shown to be the projective cover of the irreducible module $L(\uline{\la})$ in the category of $S(n\d)$-modules. This takes care of the $i=0$ case.

Now suppose that $i>0$. Since $\D(\uline{\la})$ is a direct summand of $\circ_{\w\in\W}\D(\w)^{\circ n_\w}$, it suffices to show that
\[
\Ext^i(\dct{\w\in\W}\D(\w)^{\circ n_\w},L(\uline{\mu}))=0.
\]
The module $\Res_{\d,\ldots,\d}L(\uline{\mu})$ has all composition factors a tensor product of cuspidal $R(\d)$-modules. The result now follows from adjunction and Theorem \ref{14.1}.
\end{proof}

\begin{corollary}\label{19.10}
Let $\ula$ and $\uline{\mu}$ be two multipartitions. Then $\langle [\D(\ula)],[L(\umu)]\rangle=\delta_{\ula,\umu}$.
\end{corollary}

\section{The Imaginary Part of the PBW Basis}

We now follow \cite{beckcharipressley} and define the imaginary root vectors. For comparison with their paper, we note that our $q$ is their $q^{-1}$. We will not be able to cite results from \cite{beckcharipressley} since they only work with convex orders of a particular type.
The aim of this section is to describe a purely algebraic construction of the PBW basis. We will prove that this algebraic construction agrees with the one coming from KLR algebras in Theorem \ref{pbwcat}.

Let $\om$ be a chamber coweight adapted to $\prec$. We first define elements $\psi_n^\om$ by
\[
 \psi_n^\om= E_{n\d-\om_+}E_{\om_+}-q^2 E_{\om_+} E_{n\d-\om_+}.
\]

Before we continue, we show that the $\psi_n^\om$ lie in a commutative subalgebra of $\f$.

\begin{theorem}
If $L_1$ and $L_2$ are irreducible semicuspidal representations of $R(n_1\d)$ and $R(n_2\d)$ respectively, then $L_1\circ L_2\cong L_2\circ L_1$.
\end{theorem}

\begin{proof}
The modules $L_1$ and $L_2$ are both direct summands of modules of the form $\circ_{\om} L(\om)^{\circ n_\om}$. The space of homomorphisms between two modules of this form has already been computed to be concentrated in degree zero. Therefore $\Hom(L_1,L_2)$ is concentrated in degree zero. By the same argument as in the proof of Lemma \ref{risq}, the $R$-matrices $r_{L_1,L_2}$ and $r_{L_2,L_1}$ are inverse isomorphisms.
\end{proof}

\begin{corollary}
The subalgebra of $\f$ spanned by all semicuspidal representations of $R(n\d)$ is commutative.
\end{corollary}

\begin{lemma}\label{20.3}
Let $\w\in \Om$ and $n\in\N$.
There exist semicuspidal representations $X$ and $Y$ of $R(n\d)$ with $[X]-[Y]=\psi_n^\om$. 
\end{lemma}

\begin{proof}
The same argument as in the proof of Lemma \ref{mpses} shows that we can take $X$ and $Y$ to be the cokernel and kernel of a map from $q^2\D(\om_+)\circ \D(n\d-\om_+)$ to $\D(n\d-\om_+)\circ \D(\om_+)$.
\end{proof}

\begin{corollary}\label{commute}
The elements $\psi^\w_n$ commute with each other.
\end{corollary}

Now we return to defining the imaginary part of the PBW basis and recursively define elements $P_n^\om$ by $P_0^\om=1$ and
\[
 P^\om_n=\frac{1}{[n]}\sum_{s=1}^n q^{n-s} \psi^\om_s P^\om_{n-s}.
\]

Let $\la=(\la_1\geq\la_2\geq\cdots)$ be a partition and let $t\geq \ell(\la)$ be an integer. We define
\[
 S_\la^\om = \det (P^\om_{\la_i-i+j})_{1\leq i,j\leq t}.
\]

By Corollary \ref{commute} the entries in this matrix all commute with each other so there is no ambiguity in the definition of the determinant.

The elements $P^\om_n$ here should be thought of as playing the role of the complete symmetric functions in the ring of all symmetric functions. This determinental definition shows that the elements $S^\om_\la$ are playing the role of the Schur functions. This point of view makes it clear that the definition of $S^\om_\la$ does not depend on $t$.

Let $\pi=(\b_1^{m_1},\ldots,\underline{\la},\ldots,\ga_1^{n_1})$ be a root partition. The PBW basis element $E_\pi$ is defined to be
\begin{equation}\label{pbwalg}
 E_\pi = E_{\b_1}^{(m_1)}\cdots E_{\b_k}^{(m_k)} \left( \prod_{\w\in \W}S_{\la_\w}^\w \right)E_{\ga_l}^{(m_l)} \cdots E_{\ga_1}^{(n_1)}.
\end{equation}
This agrees with the definition in \cite{becknakajima} for the special convex orders which they use.

\section{MV Polytopes}

\begin{definition}
Let $M$ be an $R(\nu)$-module. The MV polytope of $M$, denoted $P(M)$, is the convex hull of the set
\[
\{ \mu  \mid \Res_{\mu,\nu-\mu}M\neq 0\}
\]
\end{definition}

Let $\om$ be a chamber coweight. The $\om$-face of a polytope $P$ is defined to be the intersection of $P$ with the plane spanned by $\om_+$ and $\om_-$. 
For a general polytope, this construction is a cross-section. We choose to call it a face because of the following result.

\begin{proposition}
Suppose that $\om$ is adapted to the convex order $\prec$. If $L(\pi)$ is a simple module for some root partition $\pi$ such that the support of $\pi$ is contained in the span of $\om_-$ and $\om_+$, then the $\om$ face of $P(L(\pi))$ is a (possibly degenerate) 2-face of $P(L(\pi))$.
\end{proposition}

\begin{definition}
Let $\la$ be the functional on the span of $\om_-$ and $\om_+$ such that $\la(\om_+)=1$ and $\la(\om_-)=-1$.
The width of the $\om$-face of a polytope $P$ is equal to the maximum value of $\la(p)-\la(q)$ where $p$ and $q$ are two points in the $\om$-face of $P$.
\end{definition}

\begin{example}
The width of the $\om$-face of $P(L(n\d-\om_+))$ is $n$.
\end{example}

We know this because a MV polytope is completely determined by its 2-faces, which are MV polytopes for rank two root systems.

In the rest of this section, we fix a choice of chamber coweight $\w$ adapted to $\prec$. Without loss of generality, we may assume that our convex order is of the form of Example \ref{tworow}.

Note that in our labelling of the irreducible semicuspidal modules for $R(\d)$ by multipartitions, there are choices involved. Namely replacing 
$r_{L(\om),L(\om)}$ by 
its negative results in replacing the partition $\la_\w$ by its transpose. We make a choice of sign in $r_{L(\om),L(\om)}$ such that $L_\w((2))$ has $\om$-width 2.

The reason that such a choice is always possible is that the module $L(\om')\circ L(\om'')$ will only have $\w$-width at least two if $\w=\w'=\w''$ and by the Tingley-Webster classification, there exists a unique MV polytope for $2\d$ of $\w$-width 2. It must thus come from one of the summands of $L(\om)\circ L(\om)$ and we may replace our $R$-matrix with its negative if necessary to ensure that this summand is the one indexed by the partition $(2)$.

This means that the $\omega$-face of the MV polytope for $L(\om,(1^2))$ is
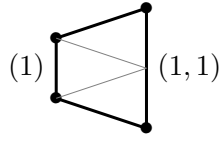
\[
\begin{tikzpicture}
 \node at (3,0){
\tikz[yscale=0.4, xscale=1.2] {
\node at (0,0) {$\bullet$}; 
\node at (0,4) {$\bullet$};
\node at (-1,1) {$\bullet$};
\node at (-1,3) {$\bullet$};
\draw [line width = 0.04cm] (0,0) -- (-1,1);
\draw [line width = 0.04cm] (-1,3) -- node[midway,left] {$(1)$}(-1,1);
\draw [line width = 0.04cm] (0,4) -- (-1,3);
\draw [line width = 0.04cm] (0,0) -- node[right,midway]{$(1,1)$}(0,4);
\draw [line width = 0.01cm, color=gray] (-1,1) -- (0, 2);
\draw [line width = 0.01cm, color=gray] (-1,3) -- (0, 2);
}
};
\end{tikzpicture}
\]

\begin{proposition}
Let $\la$ be a partition and $\w$ be a chamber coweight.
The module $L_\om(\la)$ has $\w$-width 1 if and only if $\la=(1^n)$.
\end{proposition}

\begin{proof}
We prove this proposition by an induction on $n$. The case $n=1$ is trivial and the case $n=2$ is true by the choice of normalisation of the $R$-matrix.

Note that for $\w'\neq\w$, the module $L(\w')$ has $\w$-width zero. Therefore the $\w$-width of $\dct{x\in \W} L(\la_x)$ is equal to the $\w$-width of $L(\la_w)$.

Therefore by induction we know exactly how many $\w$-faces of modules of the form $\dct{x\in\W} L(\la_x)$ with $|\la_w|<n$ have width less than or equal to one. By \cite{tingleywebster} this comprises all MV polytopes of $\w$-width less than or equal to one except for one polytope of $\w$-width one. Therefore there exists some partition $\mu\vdash n$ for which $L_\w(\mu)$ has $\w$-width one.

The restriction $\Res_{k\d,(n-k)\d}L(\mu)$ can only have composition factors $L_1\otimes L_2$ where $L_1$ and $L_2$ have $\w$-width at most one. These restrictions are given by the Littlewood-Richardson rule (\ref{lrl}). So by induction the only option is $\Res_{k\d,(n-k)\d}L(\mu)\cong L_\w(1^k)\otimes L_\w(1^{n-k})$ which for $n>2$ forces $\mu=(1^n)$, completing the proof.
\end{proof}

\begin{theorem}\label{22.6}
\[
\Res_{n\d-\om_+,\om_+} L_\om(1^n)\cong A(L_\om(1^n),L(\om_-))\otimes L(\om_+).
\]
\end{theorem}

\begin{proof}
We perform an expansion in the dual PBW basis
\[
[\Res_{n\d-\om_+,\om_+} L_\om(1^n)] =\sum_{\sigma,\pi} c_{\sigma,\pi} E_\sigma^* \otimes E_\pi^*.
\]
Then
\[
c_{\sigma,\pi} =\langle E_\sigma\otimes E_\pi,[\Res_{n\d-\om_+,\om_+} L_\om(1^n)]\rangle = \langle E_\sigma E_\pi, [L_\om(1^n)]\rangle.
\]
We consider the algorithm of \S \ref{lsformula} which teaches us how to write the product $E_\sigma E_\pi$ in terms of the PBW basis.

Let $\pi_k$ be the smallest root appearing in $\pi$. If $\pi\neq \om_+$ then as $\w_+-\pi_k\in \N I$, it must be that $\pi_k \preceq \om_-$. Therefore at all stages in applying the algorithm for writing $E_\sigma E_\pi$ in terms of the PBW basis, any term $E_{\ga_1}\cdots E_{\ga_l}$ which appears has $\ga_l\preceq \om_- \prec \d$. Therefore no purely imaginary terms in the PBW basis can appear, and as $[L_\om(1^n)]$ is orthogonal to all PBW elements which are not purely imaginary, $c_{\sigma,\pi}=0$ for such $\pi$.

So we may assume $\pi=\om_+$.

Let $\sigma_k$ be the smallest root appearing in $\sigma$. Suppose that $\sigma_k$ is not of the form $m\d+\om_-$. Then $\sigma_k\prec \om_-$. At the first stage of applying our algorithm, up to two terms $E_{\ga_1}\cdots E_{\ga_l}$ appear. One term has $\ga_l=\sigma_k\prec w_-$ while the other term, if it exists, has $\ga_l=\sigma_l+\om_+$ which is also less than $\om_-$, since by convexity it is less than $\om_+$ and we know all roots between $\om_+$ and $\om_-$. By the same argument as in the previous paragraph, $c_{\sigma,\pi}=0$ in this case too.

Therefore, when $c_{\sigma,\pi}\neq 0$, all roots that appear in $\sigma$ are all in the span of $\om_-$ and $\om_+$. This implies that every irreducible subquotient of $\Res_{n\d-\om_+,\om_+} L_\om(1^n)$ is of the form $L(\sigma)\otimes L(\om_+)$ for some such root partition $\sigma$.

The largest root appearing in $\sigma$ is at most $\d$ as $L_\om(1^n)$ is cuspidal. Therefore $\sigma=(\underline{\la},m\d-\om_+)$ for some multipartition $\underline{\la}$ and positive integer $m$.

The $\om$-face of $P(L(m\d-\om_+))$ has width $m$. Therefore the $\om$-face of $P(L(\sigma))$ has width at least $m$. As the $\om$-face of $P(L(\sigma))$ is a subset of the $\om$-face of $P(L_\om(1^n))$ which as width one, $m=1$.

Now by Theorem \ref{lrl},
\[
\Res_{(n-1)\d,\d} L_\om(1^n) \cong L_\om(1^{n-1})\otimes L(\om).
\]
Therefore the only option for $\underline{\la}$ is $1^n$ at $\omega$ and zero elsewhere, and furthermore $L(1^n_\om,\w_-)\otimes L(\om_+)$ must appear with multiplicity one, completing the proof.
\end{proof}

\begin{lemma}\label{zerocase}
Let $\w$ be a chamber coweight and $\a=\w_-$. There is a short exact sequence
\[
0\to q L(\a+\d)\to L(\w)\circ L(\a)\to L(\w,\a)\to 0.
\]
\end{lemma}

\begin{proof}
Theorem \ref{main} tells us that $L(\w,\a)$ is the head of the module $L(\w)\circ L(\a)$ and that every other subquotient of $L(\w)\circ L(\a)$ is cuspidal. Therefore there is a short exact sequence
\[
0 \to X \to L(\w)\circ L(\a)\to  L(\w,\a)\to 0
\]
for some cuspidal $R(\a+\d)$-module $X$. Since the head of $L(\w)\circ L(\a)$ is known, Lemma \ref{geometry} implies that $[X]\in q\N[q]E_{\a+\d}^*.$ Taking duals there is a short exact sequence
\[
0 \to L(\w,\a)\to L(\a)\circ L(\w)\to X\dual \to 0.
\]

We now consider
\[
\Hom(L(\a)\circ L(\w),L(\w)\circ L(\a))\cong \Hom(L(\a)\otimes L(\w), \Res_{\a,\d}L(\w)\circ L(\a)).
\]
The restriction has two nonzero pieces in its Mackey filtration. The module $L(\a)\otimes L(\w)$ appears as a quotient and we use Lemma \ref{reslw} to identify the submodule as $L(\a)\otimes (L(\d-\a)\circ L(\a))$.

Now we consider
\[
\Hom(L(\w),L(\d-\a)\circ L(\a))\cong \Hom(q^2 L(\a)\otimes L(\d-\a),L(\a)\otimes L(\d-\a))
\]
where we have used the adjunction (\ref{oppositeadjunction}) and Lemma \ref{reslw} to reach this isomorphism. Therefore there is a unique (up to scalar) morphism from $L(\a)\circ L(\w)$ to $L(\w)\circ L(\a)$ in degree 2, and the only other possible morphisms are in degree zero from the other term in the Mackey filtration. When comparing this with $[X]\in q\N[q]E_{\a+\d}^*$, the only option is that $X\cong q L(\a+\d)$, as required.
\end{proof}

\begin{lemma}\label{zerocase2}
Let $\w$ be a chamber coweight and $\a=\w_-+n\d$ for some natural number $n$. Then there are short exact sequences
\[
0\to qL(\a+\d) \to L(\w)\circ L(\a) \to L(\w,\a)\to 0 \]\[
0 \to \D(\w)\circ\D(\a) \to \D(\a) \circ \D(\w) \to q\D(\a+\d)\oplus q\inv \D(\a+\d)\to 0.
\]
\end{lemma}

\begin{proof}
We prove the existence of these short exact sequences by an induction on $n$. The case $n=0$ for the first sequence is Lemma \ref{zerocase}. First we prove the existence of the first sequence for some $n>0$, assuming that both sequences are known for lesser values of $n$.

As in the proof of Lemma \ref{zerocase}, we have a short exact sequence\[
0 \to X \to L(\w)\circ L(\a)\to  L(\w,\a)\to 0
\]
where $[X]\in q\N[q]E_{\a+\d}^*$, and we wish to study
\[
\Hom(L(\a)\circ L(\w),L(\w)\circ L(\a))\cong \Hom(L(\a)\otimes L(\w), \Res_{\a,\d}L(\w)\circ L(\a)).
\]

The Mackey filtration of $\Res_{\a,\d}(L(\w)\circ L(\a))$ has two nonzero pieces. The module $L(\a)\otimes L(\w)$ appears as a quotient, and to understand the submodule, we need to first understand $\Res_{\a-\d,\d}L(\a)$.

By Lemma \ref{oppositerestrict}, we can write
\[
[\Res_{\a-\d,\d}L(\a)] = \sum_{x\in \Om} g_x(q) [L(\a-\d)]\otimes [L(x)]
\]
for some polynomials $g_x(q)\in \N[q,q\inv]$ which satisfy $g_x(q)=g_x(q\inv)$ since restriction commutes with duality.

For $x\in \Omega$, let $C_x$ be the projective $S(\a-\d)$-module which appears in the short exact sequence of Lemma \ref{mpses}:
\[
0\to \D(x)\circ \D(\a-\d) \to \D(\a-\d)\circ \D(x)\to C_x\to 0.
\]
We compute
\begin{align*}
g_x(q)&=\langle E_{\a-\d}\otimes E_x, [\Res_{\a-\d}L(\a)] \rangle \\
&= \langle E_{\a-\d}E_x,[L(\a)] \rangle \\
&= \langle E_{\a-\d}E_x-E_x E_{\a-\d},[L(\a)] \rangle \\
&= \langle [C_x],[L(\a)]\rangle.
\end{align*}
Therefore $C_x\cong g_x(q)\D(\a)$. 

If $x\neq \w$ then we can compute the value of $[C_x]$ after specialising $q=1$ in $\f$ to obtain $g_x(1)$ is 0 or 1, which forces $g_x(q)$ to be 0 or 1.

For $x=\w$, we use the inductive hypothesis applied to the second short exact sequence to conclude that $g_\w(q)=q+q\inv$. Therefore $\Res_{\a,\d} L(\w)\circ L(\a-\d)$ has a submodule isomorphic to $q(L(\w)\circ L(\a-\d) )\otimes L(\w)$.

By the inductive hypothesis this module receives a map from $q^2 L(\a)\otimes L(\w)$ and hence there exists a morphism from $L(\a)\circ L(\w)$ to $L(\w)\circ L(\a)$ of degree two. In fact this argument shows us we know even more, namely that all other morphisms between these modules are of degree zero. So the same argument as in Lemma \ref{zerocase} allows us to conclude $X\cong qL(\a+\d)$, as required.

Now we deduce the second short exact sequence from the first. By Lemma \ref{mpses}, there exists a short exact sequence
\[
0\to \D(\w)\circ \D(\a) \to \D(\a)\circ \D(\w)\to C\to 0
\]
where $C$ is a projective $S(\a+\d)$-module, hence isomorphic to $f(q)$ copies of $\D(\a+\d)$ for some $f(q)\in\N[q,q\inv]$. The same argument computing pairings as above shows that $f(q)$ is equal to the multiplicity of $L(\a)\otimes L(\om)$ in $\Res_{\a,\d}L(\a+\d)$. The computation in $\f$ specialised at $q=1$ shows $f(1)=2$, and since $f(q)=f(q\inv)$, we have $f(q)=q^i+q^{-i}$ for some $i\in\Z$.

The first exact sequence gives us a morphism from $qL(\a+\d)$ to $L(\om)\circ L(\a)$ which by adjunction induces a nonzero morphism $\Res_{\a,\d}L(\a+\d)\to L(\a)\otimes L(\w)$. Therefore $i=1$, as required.
\end{proof}

\begin{proposition}\label{resolution}
Let $k$ and $l$ be positive integers. There is a short exact sequence
\[
0\to q A(L_\om(1^k),L((l+1)\d-\om_+))\to L_\om(1^{k+1})\circ L(l\d-\om_+)\to A(
L_\om(1^{k+1}),L(l\d-\om_+))\to 0.
\]
\end{proposition}
\begin{proof}
This proof proceeds by an induction.
By Theorem \ref{main}, the module $L_\om(1^{k+1})\circ L(l\d-\om_+)$ surjects onto $A(
L_\om(1^{k+1}),L(l\d-\om_+))$ and all other subquotients are of the form $X_{\uline{\la},m}^i=q^iA(L(\underline{\la}),L((l+m)\d-\a))$ for some $m>0$ and $\underline{\la}$ a multipartition of $k+1-m$.

Setting $n=k+1-m$, the following computation is straightforward as there is only one nonzero piece in the Mackey filtration.
\[
\Res_{n\d,(l+m)\d-\om_+} (L_\om(1^{k+1})\circ L(l\d-\om_+)  )\cong L_\om(1^n)\otimes 
\left( L_\om(1^m))\circ L(l\d-\om_+) \right)
\]
Note that if $X_{\uline{\la},m}^i$ is a subquotient of $L_\om(1^{k+1})\circ L(l\d-\om_+)$ then $L(\uline{\la})\otimes L((l+m)\d-\om_+)$ must appear as a subquotient of this restriction.
Immediately we see that $\la_\w=(1^n)$ and $\la_x=0$ for all other chamber coweights $x$.

Consider a subquotient of the form $X_{\uline{\la},m}^i$ with $\uline{\la}\neq 0$. Then by inductive hypothesis we know all that there is only a cuspidal subquotient of $ L_\om(1^m)\circ L(l\d-\om_+)$ when $m=1$. Furthermore this cuspidal subquotient appears with multiplicity $q$, which completes the proof in this case.

So now turn our attention to the remaining case when $n=0$. The module $L_\om(1^{k+1})\circ L(l\d-\om_+)$ has $\om$-width $l+1$ and the module $L((l+m)\d-\w_+)$ has $\w$-width $l+m$. Therefore $m=1$. The result now follows from Lemma \ref{zerocase2}. \end{proof}

\section{Inner Product Computations}

For any natural number $n$ and chamber coweight $\w$, define $e_n^\om=[L_\om(1^n)]$.

\begin{lemma}\label{psinen}
Let $\om$ be a chamber coweight and $\{n_x\}_{x\in \Om}$ a collection of natural numbers with sum $n$. Then
\[
\langle \psi_n^\om ,\prod_{x\in\Om}e_{n_x}^x \rangle = \begin{cases}
(-q)^{n-1} &\text{if $n_x=0$ for all $x\neq \w$,} \\
0 & \text{otherwise}.
\end{cases}
\]
\end{lemma}

\begin{proof}
By Theorem \ref{independence}, we assume without loss of generality that our convex order is as in Example \ref{tworow}.

By definition $\psi_n^\om=E_{n\d-\om_+} E_{\om_+}-q^2 E_{\om_+} E_{n\d-\om_+}$. Since $\Res_{\om_+,n\d-\om_+}L=0$ for any semicuspidal representation $L$, we have
\[
\langle \psi_n^\om ,\prod_{x\in\Om}e_{n_x}^x \rangle =\langle E_{n\d-\om_+}\otimes E_{\om_+}, \prod_{x\in\Om}r(e_{n_x}^x) \rangle.
\]

The terms in the product all commute so without loss of generality we may assume that the $r(e_{n_\om}^\w)$ term is last.

Each term appearing in the product of the $r(e_{n_x}^x)$'s is a product of terms $y\otimes z$ with $y$ of degree at most $\d$ and $z$ of degree at least $\d$. Since we need a term of degree $(n\d-\om_+,\om_+)$, the only option is that exactly one of the terms does not have degree $(n_x\d,0)$.

That particular term will have degree $(n_x-\om_+,\om_+)$. Now for $r_{n_x-\om_+,\om_+}(e_{n_x}^x)$ to not be zero, it must be that $\Res_{n_x\d-\om_+,\om_+}L_x(1^{n_x})\neq 0$ and hence the restriction $\Res_{n_x\d-\om_+,\om_+}L(x)^{\circ n_x}$ is also not zero. By a Mackey argument this implies that $\Res_{\om_-,\om_+}L(x)\neq 0$. 
By Lemma \ref{oppositerestrict} there is an injection $L(\om_-)\otimes L(\w_+)\to \Res_{\om_-,\om_+}L(x)$ and so by adjunction there is a nonzero map from $L(\om_-)\circ L(\w_+)$ to $L(x)$. By Theorem \ref{imaginaryclassification} $x=\w$.

Now Theorem \ref{22.6} and Proposition \ref{resolution} tell us that
\[
r_{n_\om\d-w_+,\om_+}(e_{n_\om}^\om) = \sum_{j=1}^{n_\om} (-q)^{j-1} e_{n-j}^\om E_{j\d-\om_+}^*\otimes E_{\w_+}^*.
\]

Therefore
\[
\langle \psi_n^\om ,\prod_{x\in\Om}e_{n_x}^x \rangle =\langle E_{n\d-\om_+},
(\prod_{ \substack{x\in\Om, \\ x\neq \w}} e_{n_x}^x )(\sum_{j=1}^{n_\w}(-q)^{j-1} e_{n-j}^\om E_{j\d-\om_+}^*)
\rangle.
\]
Since $\Res_{\d,(n-1)\d-\w_+}\D(n\d-\w_+)=0$, there is only one possible term which can be nonzero, it only occurs when $n_x=0$ for all $x\neq \w$ and $j=n_\w$. The resulting inner product is easily evaluated to $(-q)^{n-1}$.
\end{proof}

\begin{lemma}\label{pnen}
For $n\geq 0$, we have
 \[
\langle P_n^\om ,\prod_{x\in\Om}e_{n_x}^x \rangle
=\begin{cases}
             1 &\text{if $n_x=0$ for all $x\neq \w$ and $n_\w\leq 1$,}\\
0 &\text{otherwise}.
            \end{cases}
 \]
\end{lemma}

\begin{proof}
From the definition of $P_n^\w$,
\begin{align*}
\langle P_n^\om ,\prod _{x\in\Om}e_{n_x}^x \rangle &= \frac{1}{n}\sum_{s=1}^n q^{n-s} \langle 
\psi_s^\w P_{n-s}^\w, \prod _{x\in\Om}e_{n_x}^x \rangle. \\
&=\frac{1}{n}\sum_{s=1}^n q^{n-s} \langle 
\psi_s^\w \otimes P_{n-s}^\w, \prod _{x\in\Om}r(e_{n_x}^x) \rangle.
\end{align*}

Since each $L_x(1^{n_x})$ is semicuspidal, the only relevant terms in $r(e_{n_x}^x)$ are of bidegree $(k\d,l\d)$ for some $k,l\in \N$, and all these terms are known by (\ref{lrr}). Therefore
\[
\langle P_n^\om ,\prod _{x\in\Om}e_{n_x}^x \rangle = \frac{1}{n}\sum_{s=1}^n q^{n-s} \langle 
\psi_s^\w \otimes P_{n-s}^\w, \prod _{x\in\Om}\sum_{k=0}^{n_x} e_k^x \otimes e_{n_x-k}^x \rangle.
\]
This can easily be computed by an induction on $n$ together with Lemma \ref{psinen}.
\end{proof}

\section{Symmetric Functions}

Let $\La$ be the Hopf algebra of symmetric functions. We consider it over the ground ring $\Z[q,q\inv]$. It is isomorphic to $\Z[q,q\inv][h_1,h_2,\ldots]$ where $h_n$ is the complete symmetric function. Let $s_\la$ be the Schur function indexed by the partition $\la$. Let $(\cdot,\cdot)$ denote the usual inner product on $\La$ for which the Schur functions form an orthonormal basis. We denote the coproduct on $\La$ by $\D$.

Let $B$ be the subalgebra of $\f^*$ generated by the elements $e_{n}^\w$. For $x\in B$ we define $r_\d(x)\in B\otimes B$ to be the sum of all terms in $r(x)$ of bidegree $(a\d,b\d)$.

\begin{lemma}\label{hopfiso}
There is an isomorphism of Hopf algebras $\psi\map{\La^{\otimes \W}}{B}$ with
\[
\psi(\otimes_\w s_{\la_w}) = [L(\ula)]
\] where the coproduct on $B$ is $r_\d$.
\end{lemma}

\begin{proof}
This is immediate from Theorem \ref{lrl}.
\end{proof}

Define an algebra homomorphism $\varphi\map{\La^{\otimes \W}}{\f}$ by
\[
\varphi (\otimes_\w h_{n_\w}) = \prod_{\w\in\W} P^\w_{n_\w}.
\]
That such a homomorphism exists is because the $h_{n_\w}$ freely generate $\La^{\otimes \W}$ as a commutative algebra and Corollary \ref{commute} which implies that the $P_{n_\w}^\w$ lie in a commutative subalgebra of $\f$.

\begin{lemma}\label{comp}
For all $x,y\in\La^{\otimes \W}$ we have
\[
\langle \varphi(x),\psi(y) \rangle = (x,y).
\]
\end{lemma}

\begin{proof}
Lemma \ref{pnen} establishes this formula in the special case when $x=P_n^\w$. To deduce the general case from this particular case, we use $(xy,z)=(x\otimes y,z)$ and
$$\langle \varphi(xy),\psi(z)\rangle = \langle \varphi(x) \varphi(y),\psi(z)\rangle = \langle \varphi(x) \otimes \varphi(y), r_\d (\psi(z)) \rangle = \langle \varphi(x) \otimes \varphi(y),\psi(\D(z))\rangle $$
where in the last step we used Lemma \ref{hopfiso}.
\end{proof}

\begin{corollary}\label{pairsandl}
Let $\om$ and $\om'$ be two chamber coweights and let $\la$ and $\mu$ be partitions. Then $\langle S_\la^\om,[L_{\om'}(\mu)] \rangle = \delta_{\om\om'}\delta_{\la\mu}$.
\end{corollary}

\begin{proof}
The Schur functions are orthonormal.
\end{proof}



\begin{theorem}\label{impbw}
Let $\ula=\{\la_\w\}_{\w\in\W}$ be a purely imaginary root partition. Then
$$[\D(\ula)]=\prod_{w\in\W}S^\w_{\la_\w}$$
\end{theorem}

\begin{proof}
The nondegeneracy of $(\cdot,\cdot)$ together with Lemma \ref{comp} implies that $\varphi$ is injective. By Lemmas \ref{20.3} and \ref{10.2}, the image of $\varphi$ lies in the subspace of $\f_{\Z((q))}^*$ spanned by the semicuspidal modules. A dimension count shows that the image is precisely the span of the semicuspidal modules.
Therefore $\D_\w(\ula)$ is a linear combination of the elements $S_{\uline{\mu}}$. 

The pairings in Corollary \ref{19.10} and \ref{pairsandl} force $\D(\uline{\la})=S_{\uline{\la}}$.
\end{proof}


\section{Standard Modules}

The nil Hecke algebra $NH_n$ is the algebra $R(ni)$ for any $i\in I$.  
It is well known that the nil Hecke algebra is a matrix algebra over its centre, see for example \cite[Proposition 2.21]{rouquier2}. In particular, there is an isomorphism
\[
NH_n \cong \Mat_{[n]!}(\Q[x_1,\ldots,x_n]^{S_n})
\]
where each $x_i$ is in degree two.

Let $e_n$ be a primitive idempotent in $NH_n$.

\begin{theorem}
 Let $\a$ be a real root. There is an isomorphism $\End(\D(\a)^{\circ n})\cong NH_n$.
\end{theorem}

\begin{proof}
 The proof of \cite[\S 3]{bkm} works in this generality without any change.
\end{proof}

For any positive real root $\a$ and any positive integer $n$, we define the divided power standard module $\D(\a)^{(n)}$ to be
\[
 \D(\a)^{(n)} =q^{n(n-1)/2} e_n (\D(\a)^{\circ n})
\]

\begin{lemma}\label{24.3}
Let $\a$ be a real root and $n$ a positive integer. Then
\[
\Ext^i(\D(\a)^{(n)},L(\a)^{\circ n}) \cong \begin{cases}
\Q \text{ if $i=0$} \\
0 \text{ otherwise}
\end{cases}
\]
\end{lemma}

\begin{proof}
We compute by adjunction
\[
\Ext^i(\D(\a)^{\circ n},L(\a)^{\circ n})\cong \Ext^i(\D(\a)^{\otimes n},\Res_{\d,\ldots,\d}L(\a)^{\circ n}).
\]
The module $\Res_{\d,\ldots,\d}L(\a)^{\circ n}$ has a composition series with $n!$ subquotients, each isomorphic to some $q^jL(\a)^{\otimes n}$ and $[\Res_{\d,\ldots,\d}L(\a)^{\circ n}]=[n]!E_\a^*\otimes\cdots\otimes E_\a^*$. So by Theorem \ref{14.1}, for $i>0$ we have
\[
\Ext^i(\D(\a)^{\circ n},L(\a)^{\circ n})=0
\]
while for $i=0$ we also use the fact that $\D(\a)$ is the projective cover of $L(\a)$ in the category of $S(\a)$-modules to obtain
\[
\Hom(\D(\a)^{\circ n},L(\a)^{\circ n})\cong q^{{n \choose 2}} [n]!\Q.
\]
Since $\D(\a)^{\circ n}\cong  q^{{n \choose 2}} [n]!\D(\a)^{(n)}$, we obtain the desired result.
\end{proof}

Let $\pi=(\b_1^{m_1},\ldots,\b_k^{m_k},\underline{\la},\ga_l^{n_l},\ldots,\ga_1^{n_1})$ be a root partition. We define the corresponding standard module to be
\[
 \D(\pi)=\D(\b_1)^{(m_1)}\circ\cdots\circ\D(\b_k)^{(m_k)}\circ \D(\ula)\circ\D(\ga_l)^{(n_l)}\circ\cdots\circ 
\D(\ga_1)^{(n_1)}.
\]
Also define
\[
 \nabla(\pi)=\overline{\D}(\pi)\dual.
\]

\begin{proposition}\label{homological}
Let $\pi$ and $\sigma$ be two root partitions. Then
\[
 \Ext^i(\D(\pi),\nabla(\sigma)) =\begin{cases}
             \Q &\text{if } i=0 \text{ and } \pi=\sigma \\
0 &\text{otherwise}
            \end{cases}
\]
\end{proposition}

\begin{proof}
 Let $\pi=(\b_1^{m_1},\cdots,\ga_1^{n_1})$. Then by adjunction,
\[
 \Ext^i(\D(\pi),\co(\sigma))=\Ext^i(\D(\b_1)^{(m_1)}\otimes\cdots \otimes \D(\ga_1)^{(n_1)},\Res_\pi\co(\sigma))
\]
and by Theorem \ref{main}, $\Res_\pi\co(\sigma)=0$ unless $\pi\leq\sigma$.

On the other hand, by the adjunction (\ref{oppositeadjunction}),
\[
 \Ext^i(\D(\pi),\co(\sigma)) = \Ext^i ( \Res_\sigma \D(\pi),L(\ga_1)^{\otimes n_1}\otimes \cdots \otimes L(\b_1)^{\otimes m_1})
\]
and so again using Theorem \ref{main}, $\Res_\sigma \D(\pi)=0$ unless $\sigma\leq \pi$.

Thus the only case to consider is when $\sigma\sim\pi$. Remember that this means that $\sigma$ and $\pi$ agree except for the multipartition they contain. Let $\ula$ be the multipartition in $\pi$ and $\umu$ be the multipartition in $\sigma$.

By Theorem \ref{main}, $$\Res_\pi\co(\pi)\cong \Res_\p(\overline{\D}(\p)\dual)\cong (\Res_\p \overline{\D}(\p))\dual \cong L(\b_1)^{\circ m_1}\otimes \cdots \otimes L(\ga_1)^{\circ n_1}.$$ Therefore 
\[
\Ext^*(\D(\pi),\nabla(\sigma))\cong \bigotimes_\a \Ext^*(\D(\a)^{(f_\p(\a))},L(\a)^{\circ f_\sigma(\a)}) \otimes \Ext^*(\D(\ula),L(\uline{\mu}))
\] where the tensor product is over all real roots $\a$.

The result now follows from Lemma \ref{24.3} and Theorem \ref{14.1nd}.
\end{proof}

\begin{theorem}\label{pbwcat}
Let $\p$ be a root partition.
The class of the standard module $\D(\p)$ is the PBW monomial $E_\p$, defined algebraically in (\ref{pbwalg}).
\end{theorem}

\begin{proof}
 This follows from Theorems \ref{realpbw} and \ref{impbw}.
\end{proof}

Proposition \ref{homological} proves that the classes of the standard modules $\D(\pi)$ and the proper standard modules $\overline{\D}(\sigma)$ are orthogonal under $(\cdot,\cdot)$. Therefore the class of each proper standard module is an element of the dual PBW basis. So we have categorified both the PBW and dual PBW basis.

A module $M$ is said to have a $\D$-flag if it has a sequence of submodules $0=M_0\subseteq M_1\subseteq \cdots
\subseteq M_{n-1}\subseteq M_n=M$ such that each subquotient $M_{i+1}/M_i$ is isomorphic to $q^m\D(\pi)$ for some integer $m$ and some root partition $\pi$.

\begin{theorem}\label{deltaflag}
Let $M$ be a finitely generated $R(\nu)$-module such that $\Ext^1(M,\nabla(\pi))=0$ for all root partitions $\pi$. Then $M$ has a $\D$-flag. Furthermore $[M:\D(\pi)]=\dim\Hom(M,\nabla(\pi))$.
\end{theorem}

\begin{proof}
 This is a standard argument, for example see \cite[Theorem 3.13]{bkm}.
\end{proof}

As a consequence we obtain the following BGG reciprocity for KLR algebras.

\begin{theorem}\label{15.4}
Let $\pi$ be a root partition and let $\P(\pi)$ be the projective cover of $L(\pi)$. Then $\P(\pi)$ has a $\D$-flag. For any root partition $\sigma$ the multiplicity $[\P(\pi):\D(\sigma)]$ is equal to the multiplicity $[\overline{\D}(\sigma):L(\pi)]$.
\end{theorem}

\begin{proof}
Since $\P(\pi)$ is finitely generated and projective it satisfies the hypotheses of Theorem \ref{deltaflag} and hence has a $\D$-flag. Furthermore the multiplicity of the module $\D(\sigma)$ in the flag is
\[
 [\P(\pi):\D(\sigma)]=\dim\Hom(P(\pi),\co(\sigma)).
\]
As $\P(\pi)$ is the projective cover of $L(\pi)$, the dimension of this homomorphism space is equal to the multiplicity $[\nabla(\sigma)\dual:L(\pi)]$. By duality
$[\nabla(\sigma):L(\pi)]=[\overline{\D}(\sigma):L(\pi)\dual]$ and since $L(\pi)\cong L(\pi)\dual$ we are done.
\end{proof}

\begin{theorem}
The PBW basis (\ref{pbwalg}) is a basis of $\f$ as a $\Z[q,q\inv]$-module.
\end{theorem}

\begin{proof}
 By Theorem \ref{15.4} and Theorem \ref{main}(3), the matrix expressing the set $\{[\D(\pi)]\}$ in terms of the basis $\{[\P(\pi)]\}$ is upper-triangular, with ones along the diagonal.
Therefore the set $\{[\D(\pi)]\}$ is a basis of $\f$ as a $\Z[q,q\inv]$-module.
\end{proof}

\begin{remark}
 This is a generalisation, with a different proof, of a result of \cite{becknakajima}.
\end{remark}

\begin{proposition}
 With respect to the PBW basis, the bar involution is unitriangular.
\end{proposition}

\begin{proof}
By Proposition \ref{homological}, the PBW basis is dual to the basis $[\nabla(\pi)]$ under the pairing $\langle\cdot,\cdot\rangle$. It suffices to prove that the bar involution on $\f^*$ is unitriangular with respect to this basis. Since each $\nabla(\pi)$ is an induction product of self-dual simples up to an overall grading shift, it is easy to see that the bar involution is unitriangular by Theorem \ref{main}(3).
\end{proof}

Once we have that the bar-involution is unitriangular, it is straightforward to show that there exists a unique basis $b_\p$ of $\f$ which is bar-invariant and for which
\[
 b_\p=E_\p + \sum_{\sigma<\pi} c_{\p\sigma} E_\sigma
\]
where $c_{\p\sigma}\in q\Z[q]$. Theorem \ref{positif} below shows that the basis $\{b_\p\}$ is the canonical basis, providing an algebraic characterisation of the canonical basis.

Thus from Theorem \ref{15.4} and the fact that the indecomposable projective modules categorify the canonical basis, we obtain the following positivity result.
\begin{theorem}\label{positif}
The change of basis matrix from the canonical basis to a PBW basis is unitriangular with off diagonal entries lying in $q\N[q]$.
\end{theorem}

\begin{proof}
 The fact that the coefficients are all nonnegative is from Theorem \ref{15.4} and the fact that the indecomposable projective modules categorify the canonical basis. That the coefficients lie in $q\Z[q]$ follows from Lemma \ref{geometry}.
\end{proof}

This positivity result is new in affine type. In finite type this result is \cite[Corollary 10.7]{lusztigoneofthem} for particular convex orders and for all convex orders is due to Kato and the author independently in \cite{kato,klr1}.

\def\cprime{$'$}


\begin{thebibliography}{KMR12}

\bibitem[BBD]{bbd}
A.~A. Be{\u\i}linson, J.~Bernstein, and P.~Deligne.
\newblock Faisceaux pervers.
\newblock In {\em Analysis and topology on singular spaces, {I} ({L}uminy,
  1981)}, volume 100 of {\em Ast\'erisque}, pages 5--171. Soc. Math. France,
  Paris, 1982.

\bibitem[BCP]{beckcharipressley}
Jonathan Beck, Vyjayanthi Chari, and Andrew Pressley.
\newblock An algebraic characterization of the affine canonical basis.
\newblock {\em Duke Math. J.}, 99(3):455--487, 1999.

\bibitem[Bec]{beck}
Jonathan Beck.
\newblock Convex bases of {PBW} type for quantum affine algebras.
\newblock {\em Comm. Math. Phys.}, 165(1):193--199, 1994.

\bibitem[BKM]{bkm}
Jonanthan Brundan, Alexander Kleshchev, and Peter~J. McNamara.
\newblock Homological properties of finite type {K}hovanov-{L}auda-{R}ouquier
  algebras.
\newblock {\em Duke Math. J.}, 163(7):1353--1404, 2014.

\bibitem[BN]{becknakajima}
Jonathan Beck and Hiraku Nakajima.
\newblock Crystal bases and two-sided cells of quantum affine algebras.
\newblock {\em Duke Math. J.}, 123(2):335--402, 2004.

\bibitem[CP]{cellinipapi}
Paola Cellini and Paolo Papi.
\newblock The structure of total reflection orders in affine root systems.
\newblock {\em J. Algebra}, 205(1):207--226, 1998.

\bibitem[Ito]{ito}
Ken Ito.
\newblock The classification of convex orders on affine root systems.
\newblock {\em Comm. Algebra}, 29(12):5605--5630, 2001.

\bibitem[Kat]{kato}
Syu Kato.
\newblock {PBW} bases and {KLR} algebras.
\newblock {\em Duke Math. J.}, 163(3):619--663, 2014.

\bibitem[KKK]{kkk}
Seok-Jin Kang, Masaki Kashiwara, and Myungho Kim.
\newblock Symmetric quiver {H}ecke algebras and {R}-matrices of quantum affine
  algebras.
\newblock {\em \arxiv{1304.0323}}.

\bibitem[KKKO]{kkko}
Seok-Jin Kang, Masaki Kashiwara, Myungho Kim, and Se-Jin Oh.
\newblock Simplicity of heads and socles of tensor products.
\newblock {\em Compos. Math.}, 151(2):377-396, 2015.

\bibitem[KL1]{khovanovlauda}
Mikhail Khovanov and Aaron~D. Lauda.
\newblock A diagrammatic approach to categorification of quantum groups. {I}.
\newblock {\em Represent. Theory}, 13:309--347, 2009.

\bibitem[KL2]{kl2}
Mikhail Khovanov and Aaron~D. Lauda.
\newblock A diagrammatic approach to categorification of quantum groups {II}.
\newblock {\em Trans. Amer. Math. Soc.}, 363(5):2685--2700, 2011.

\bibitem[Kle]{kleshchev}
Alexander Kleshchev.
\newblock Cuspidal systems for affine {K}hovanov-{L}auda-{R}ouquier algebras.
\newblock {\em Math. Z.}, 279(3-4):691--726, 2014.

\bibitem[KM]{imaginaryschurweyl}
Alexander Kleshchev and Robert Muth.
\newblock Imaginary {S}chur-{W}eyl duality.
\newblock {\em \arxiv{1312.6104}}.

\bibitem[KMR]{kmr}
Alexander~S. Kleshchev, Andrew Mathas, and Arun Ram.
\newblock Universal graded {S}pecht modules for cyclotomic {H}ecke algebras.
\newblock {\em Proc. Lond. Math. Soc. (3)}, 105(6):1245--1289, 2012.

\bibitem[KR]{kleschevram}
Alexander Kleshchev and Arun Ram.
\newblock Representations of {K}hovanov-{L}auda-{R}ouquier algebras and
  combinatorics of {L}yndon words.
\newblock {\em Math. Ann.}, 349(4):943--975, 2011.

\bibitem[LS91]{ls}
Serge Levendorski{\u\i} and Yan Soibelman.
\newblock Algebras of functions on compact quantum groups, {S}chubert cells and
  quantum tori.
\newblock {\em Comm. Math. Phys.}, 139(1):141--170, 1991.

\bibitem[Lus1]{lusztigoneofthem}
G.~Lusztig.
\newblock Canonical bases arising from quantized enveloping algebras.
\newblock {\em J. Amer. Math. Soc.}, 3(2):447--498, 1990.

\bibitem[Lus2]{lusztigbook}
George Lusztig.
\newblock {\em Introduction to quantum groups}, volume 110 of {\em Progress in
  Mathematics}.
\newblock Birkh\"auser Boston Inc., Boston, MA, 1993.

\bibitem[LV]{laudavazirani}
Aaron~D. Lauda and Monica Vazirani.
\newblock Crystals from categorified quantum groups.
\newblock {\em Adv. Math.}, 228(2):803--861, 2011.

\bibitem[Mak]{maksimau}
Ruslan Maksimau.
\newblock Canonical basis, {KLR}-algebras and parity sheaves.
\newblock {\em J. Algebra}, 422:563--610, 2015.

\bibitem[McN]{klr1}
Peter McNamara.
\newblock Finite dimensional representations of {K}hovanov-{L}auda-{R}ouquier
  algebras {I}: {F}inite type.
\newblock {\em \arXiv{1207.5860}}.

\bibitem[Rou1]{rouquier}
R.~Rouquier.
\newblock 2-{K}ac-{M}oody algebras.
\newblock {\em \arXiv{0812.5023}}.

\bibitem[Rou2]{rouquier2}
Rapha\"el Rouquier.
\newblock {Quiver Hecke algebras and 2-Lie algebras.}
\newblock {\em Algebra Colloq.}, 19(2):359--410, 2012.

\bibitem[TW]{tingleywebster}
Peter Tingley and Ben Webster.
\newblock Mirkovic-{V}ilonen polytopes and {K}hovanov-{L}auda-{R}ouquier
  algebras.
\newblock {\em \arxiv{1210.6921}}.

\bibitem[VV]{vv}
M.~Varagnolo and E.~Vasserot.
\newblock Canonical bases and {KLR}-algebras.
\newblock {\em J. Reine Angew. Math.}, 659:67--100, 2011.

\end{thebibliography}
\end{document}